\documentclass[authoryear,11pt,review]{elsarticle}
\usepackage[margin=1in]{geometry}
\usepackage[hidelinks]{hyperref}

\usepackage{amsmath, amssymb, amsfonts, amsthm}

\usepackage[english]{babel}
\usepackage[utf8]{inputenc}
\usepackage[T1]{fontenc}

\usepackage[authoryear]{natbib}

\usepackage{eucal, tikz}

\newcommand{\bbE}{\mathbb{E}}
\newcommand{\bbN}{\mathbb{N}}
\newcommand{\bbP}{\mathbb{P}}
\newcommand{\bbR}{\mathbb{R}}

\newcommand{\one}{\boldsymbol{1}}

\newcommand{\nint}[1]{\left\lfloor #1 \right\rceil}

\newcommand{\dd}{\mathrm d}

\newtheorem{theorem}{Theorem}[section]

\theoremstyle{definition}
\newtheorem{remark}[theorem]{Remark}

\begin{document}

\begin{frontmatter}



\title{Input estimation from discrete workload observations\\ in a Lévy-driven storage system\footnote{To appear in Statistics and Probability Letters.}}

\author[1]{Dennis Nieman}
\author[2]{Michel Mandjes}
\author[3]{Liron Ravner}
\address[1]{\small{Department of Mathematics, Vrije Universiteit Amsterdam}}
\address[2]{\small{Korteweg-de Vries Institute, University of Amsterdam}}
\address[3]{\small{Department of Statistics, University of Haifa}}

\begin{abstract}
Our goal is to estimate the characteristic exponent of the input to a Lévy-driven storage system  from a sample of equispaced workload observations. The estimator relies on an approximate moment equation associated with the Laplace-Stieltjes transform of the workload at exponentially distributed sampling times. The estimator is pointwise consistent for any observation grid.  Moreover, a high frequency sampling scheme yields asymptotically normal estimation errors for a class of input processes.  A resampling scheme that uses the available information in a more efficient manner is suggested and assessed via simulation experiments.

\end{abstract}


\end{frontmatter}

\noindent
{\it Acknowledgments.} The research of Michel Mandjes and Liron Ravner is partly funded by NWO Gravitation project N{\sc etworks}, grant number 024.002.003. 

\section{Introduction}
To optimally design and control queueing systems, it is of crucial importance to have reliable estimates of model primitives. One important class of such systems is formed by so called {\it L\'evy-driven queues}, see e.g.\ \cite{dm2015}, which can be regarded as L\'evy processes on which the Skorokhod map is imposed (or, equivalently, L\'evy processes reflected at $0$). A L\'evy process is uniquely characterized by its Lévy exponent, a function that captures a full probabilistic description of the process dynamics. This paper deals with estimation of the characteristic L\'evy exponent from equidistant observations of the reflected process. We focus on storage systems whose input is a non-decreasing Lévy process. A special case of the last type of systems is the classical M/G/1 queue, where the driving L\'evy process is a compound Poisson process. 

When it comes to estimation of model primitives in a queueing system, a main complication is that in many situations one cannot observe the queue’s input process (in the concrete case of the M/G/1 queue: the individual interarrival times and service requirements); instead, one often only has discrete equidistant observations of the associated  workload process. The challenge lies in developing sound techniques to statistically estimate, based on these workload observations, the L\'evy exponent corresponding the system's input process. This type of inverse problem is challenging as densities are unavailable, impeding the use of conventional maximum likelihood procedures. In \cite{rbm2019} a method of moments estimator was constructed for the L\'evy exponent of the input process by sampling the workload according to an independent Poisson process. The advantage of so-called Poisson sampling is that the distribution of the workload after an exponentially distributed time conditional on an initial workload is known, in terms of the Laplace-Stieltjes transform (LST). This paper shows that the Poisson sampling framework can be leveraged to construct estimators for the L\'evy exponent even when the sampling times are deterministic (rather than Poisson).

\subsection{Approach and main contributions}
The main contribution lies in the construction of an estimator of the L\'evy exponent of the queue's input process, based on discrete equidistant workload observations. To this end we use an approximate estimation equation that relies on the Poisson sampling scheme of \cite{rbm2019}. 

A second contribution concerns an intermediate step in the approach of \cite{rbm2019}: there the approach required the estimation of the inverse of the  L\'evy exponent at the point of the Poisson sampling rate, whereas we succeed in providing a modified estimator that no longer requires this. On top of the practical and computational advantages of this direct method, this enables a simpler derivation of the asymptotic variance of the estimation error.

Third, we provide performance guarantees of the resulting estimator.  We start by establishing consistency, in which the mesh size $\Delta$ (the time between two subsequent observations) does not play any role. We then prove asymptotic normality, which requires us to pick $\Delta$ in a specific way, depending on the number of observations $n$; the underlying argumentation distinguishes between the case that the driving L\'evy process is of finite and infinite intensity.  For the case of infinite intensity the CLT requires an additional condition that the Blumenthal-Getoor index $\beta$ (see \cite{bg1961}) satisfies $\beta<1/4$. 

Finally, we propose a resampling scheme that uses the available information in a more efficient manner. Its performance is studied in a simulation experiment.

\subsection{Related literature}
An exhaustive recent overview of the existing statistical queueing literature is given in \cite{ant2021}. Here we restrict ourselves to discussing a few results that directly to our work.  %Namely, indirect estimation of input primitives and  their performance guarantees such as consistency and asymptotic normality of the estimation errors.
 The work of \cite{bnl1996} discusses a maximum likelihood estimator based on the waiting times of the customers in a GI/G/1 queue.   For the M/G/1 queue,  \cite{hp2006} construct a non-parametric approach for the job-size distribution from discretely observed workload observations.  Another stream of the literature focuses on nonparametric estimation of the input distribution to infinite server queueing systems; these systems have the intrinsic advantage that clients do no wait, which is, for estimation purposes, a convenient feature. For example, \cite{bp1999} and \cite{g2016} suggest estimators based on observations of the number of clients present.

There is also substantial literature on estimation of L\'evy processes (i.e., without reflection). Here, having discrete observations (of the process' increments, that is) is also a common assumption.  We refer to for example  \cite{k2010}, \cite{br2015} and references therein. \cite{gs2010} propose an estimator for the L\'evy triplet based on a high-frequency sampling scheme and show that the asymptotic performance for the finite intensity case (i.e., compound Poisson) differs from that of the infinite intensity case. Namely, consistency holds for both but asymptotic normality of the estimation errors only for the finite intensity case. This distinction is crucial in our work as well, but in our setup asymptotic normality is established for infinite intensity input processes that satisfy an additional condition on the intensity index defined by \cite{bg1961}. 

No statistical methods had been proposed in the setting of storage systems with Lévy input, until the appearance of \cite{rbm2019}. This work suggests an estimator of the L\'evy exponent based on workload observations at Poisson epochs --- a method known as `Poisson probing'; see e.g.\ \cite{bkv2009}.   In \cite{mr2021} the setup of \cite{rbm2019} is further used for constructing hypothesis tests for L\'evy-driven systems.  The estimator of the L\'evy exponent of the input process is leveraged  in \cite{r2023} to yield a non-parametric estimator of the job-size distribution in an M/G/1 queue.   %The current work facilitates the direct extension of the framework of \cite{rbm2019}  to the very natural assumption that the workload is observed at equidistant times. 

%
%\subsubsection*{Organization} This paper has been organized as follows. In Section \ref{s:model} we provide a detailed description of the model considered, including the observation mechanism. Then Section \ref{s:estimator} construct the estimator, the consistency of which is established in Section \ref{sec:cons} and the asymptotic normality in Section \ref{sec:clt}. Finally, Section \ref{sec:sim} numerically demonstrates the efficacy of the approach, in particular detailing a resampling scheme that efficiently exploits the available workload observations.

\section{Model}\label{s:model}
To introduce our model, we briefly recall the 
construction of the storage model under consideration, also known as a reflected L\'evy process. Many facts will be stated without proof or reference; see for background material e.g. the books \cite{dm2015} and  \cite{p1998}.

Let $J(\cdot) = \{J(t)\}_{t \geq 0}$ be an almost surely increasing Lévy process (a \textit{subordinator}). The Lévy process assumption entails that $J(\cdot)$ has stationary, independent increments and càdlàg sample paths, and starts at zero. The process $J(\cdot)$ is considered the input to a storage system with unit rate linear output. We consider the net input process $X(\cdot)$ whose value at time $t$ is $X(t) = J(t) - t$. The workload process associated with the model is defined through 
$V(t) = X(t) + \max\{V(0), -\inf_{0 \leq s \leq t} X(s)\}$.
We say that $V(\cdot)$ is the reflection of $X(\cdot)$ at $0$.
We impose the stability condition $\bbE X(1) < 0$, so that that the Lévy exponent $\varphi : [0,\infty) \to \bbR$ given by $\varphi(\alpha) = \log \bbE e^{-\alpha X(1)}$, is finite and strictly increasing on $[0,\infty)$. As such, $\varphi$ has an inverse, denoted by $\psi$. 
These assumptions imply that the L\'evy exponent is necessarily of the form
\begin{equation}\label{e:levyk}
\varphi(\alpha) = \alpha - \int_{(0,\infty)}\big( 1-e^{-\alpha x} \big)\,\nu(\dd x),
\end{equation}
for some unique measure $\nu$ on $(0,\infty)$ satisfying $\int_{(0,\infty)} \min\{1,x\} \,\nu(\dd x) < \infty$. %The integral term in the above display is the Lévy exponent of $J(\cdot)$.

%The Lévy exponent and its inverse play an important role in the analysis of the storage model. 
If $T$ is exponentially distributed with rate  $\xi>0$, independent of $X(\cdot)$, and if $V(0)=x$, then
\begin{equation} \label{e:lapVT}
\bbE_x e^{-\alpha V(T)} = \int_0^\infty \xi e^{-\xi t} \bbE_x e^{-\alpha V(t)} \,\dd t = \frac{\xi}{\xi-\varphi(\alpha)} \Big( e^{-\alpha x} - \frac{\alpha}{\psi(\xi)} e^{-\psi(\xi)x}\Big);
\end{equation}
see for instance \cite{kbm2006}; the dependence on the initial workload $V(0)=x$ is indicated by the subscript $x$. In the sequel, we also use the identity, following  upon combining \eqref{e:levyk} and \eqref{e:lapVT},
\begin{equation} \label{e:zeroVT}
\bbP_x(V(T) = 0) = \lim_{\alpha\to\infty} \bbE_x e^{-\alpha V(T)} = \frac{\xi}{\psi(\xi)} e^{-\psi(\xi)x}.
\end{equation}

The stability condition also implies existence of a stationary distribution $\pi$ for $V(\cdot)$. It is the unique distribution satisfying, with $A\subseteq [0,\infty)$,
$\pi(A) = \int_{[0,\infty)} \bbP_v(V(t) \in A) \,\pi(\dd v),\quad t\geq0$.
The stationary distribution $\pi$ is also the weak limit of $V(t)$ as $t\to\infty$. We will use the so-called generalized Pollaczek-Khintchine formula (following from \eqref{e:lapVT} by letting $\xi\downarrow 0$)
\begin{equation}\label{e:piexp}{\mathbb E}_x\,e^{-\alpha V(\infty)}=
\int_{[0,\infty)} e^{-\alpha v} \,\pi(\dd v) = \frac{\alpha \varphi'(0)}{\varphi(\alpha)},
\end{equation}
with the left-most expression obviously not depending on the initial storage level $x$. 
From this it also follows, with help of \eqref{e:levyk}, that the stationary probability of zero storage equals
\begin{equation}\label{e:pizero}
\pi(\{0\})=\lim_{\alpha\to\infty} \int_{[0,\infty)} e^{-\alpha v} \,\pi(\dd v) = \varphi'(0). \qedhere
\end{equation}

%The Lévy exponent $\varphi(\cdot)$ is the model primitive of interest. In the next section, we construct an estimator of it based on the discrete observations $\{ V(i\Delta), i=0,1,\ldots \}$ for some grid width $\Delta > 0$. 

\section{Estimator of the Lévy exponent}\label{s:estimator}
%Using a method of moments procedure similar to the one used in \cite{rbm2019}, we construct an estimatorf or the L\'evy exponent $\varphi(\alpha)$. 
We start by observing the storage level at Poisson instants, which will then be rounded off to the nearest multiple of the grid width $\Delta>0$. Let $\xi > 0$ and consider the sequence $T_1,T_2,\ldots$ which are i.i.d.\ exponentially distributed random variables with rate $\xi$. Denote its partial sums by $S_i = \sum_{j=1}^i T_j$. We thus obtain a sequence of observations $V_i := V(S_i)$, with each $V_i$ occurring an exponentially distributed time after its precursor $V_{i-1}$. By \eqref{e:lapVT} and \eqref{e:zeroVT},
\begin{align}\nonumber
\bbE[e^{-\alpha V_i} | V_{i-1}] &= \frac{\xi}{\xi-\varphi(\alpha)} \Big( e^{-\alpha V_{i-1}} - \frac{\alpha}{\psi(\xi)} e^{-\psi(\xi)V_{i-1}}\Big) \\
&= (\xi-\varphi(\alpha))^{-1}\Big( \xi e^{-\alpha V_{i-1}} - \alpha\, \bbP(V_i=0\,|\,V_{i-1})\Big).\label{e:lapVi}
\end{align}
Rearranging and taking expectations in \eqref{e:lapVi} yields
\begin{equation*}
\varphi(\alpha) \,\bbE e^{-\alpha V_i} = \xi \,\bbE\big(e^{-\alpha V_i} - e^{-\alpha V_{i-1}}\big) + \alpha \,\bbP(V_i = 0).
\end{equation*}
The corresponding empirical moment equation is
\begin{equation*}
\varphi(\alpha) \frac{1}{n} \sum_{i=1}^n e^{-\alpha V_i} = \xi \frac{1}{n} \sum_{i=1}^n (e^{-\alpha V_i} - e^{-\alpha V_{i-1}}) + \alpha \frac{1}{n} \sum_{i=1}^n \one\{V_i = 0\}.
\end{equation*}
Solving for $\varphi(\alpha)$, and recognizing a telescoping sum, we thus obtain an estimator for $\varphi(\alpha)$:
\begin{equation}\label{e:PPest}
\hat{\varphi}_n(\alpha) = \frac{n^{-1}\xi(e^{-\alpha V_n} - e^{-\alpha V_0}) + \alpha n^{-1} \sum_{i=1}^n \one\{V_i=0\}}{n^{-1} \sum_{i=1}^n e^{-\alpha V_i}}.
\end{equation}
The factors $n^{-1}$ in \eqref{e:PPest} are kept since they will be used in the asymptotic analysis later. Note that this estimator, which is (to the best of our knowledge) novel, sidesteps the intermediate step of estimating $\psi(\xi)$, which was necessary in the method of \cite{rbm2019}.

We then convert the estimator \eqref{e:PPest} into an estimator that only uses observations on the grid $0, \Delta, 2\Delta,\ldots,$ applying a specific discretization. To this end, let $S_i^\Delta = \Delta \nint{S_i/\Delta}$, $V_i^\Delta = V(S_i^\Delta)$, where $\nint{\,\cdot\,}$ is the nearest integer function, so that $\Delta\nint{t/\Delta}$ is the integer multiple of $\Delta$ closest to any fixed $t \geq 0$. Correspondingly, we define the `$\Delta$-counterpart' of \eqref{e:PPest}:
\begin{equation}\label{e:est}
\hat{\varphi}_n^\Delta(\alpha) = \frac{n^{-1}\xi(e^{-\alpha V_n^\Delta} - e^{-\alpha V_0^\Delta}) +\alpha n^{-1} \sum_{i=1}^n \one\{V_i^\Delta=0\}}{n^{-1} \sum_{i=1}^n e^{-\alpha V_i^\Delta}}.
\end{equation}
The observations $V_i^\Delta$ form a random subset of the equidistant observations $\{V(i\Delta) : i=0,1,\ldots\}$. An illustration of this sampling scheme is given in Figure~\ref{f:sampling}. Observe that it is possible, for small $T_i$, that $S_i^\Delta=S_{i-1}^\Delta$. In this case there is no additional information from sample $i$, however the estimator \eqref{e:est} is still well defined and this will not impact its asymptotic performance.

\begin{figure*}[h]
	\makebox[\textwidth][c]{\includegraphics{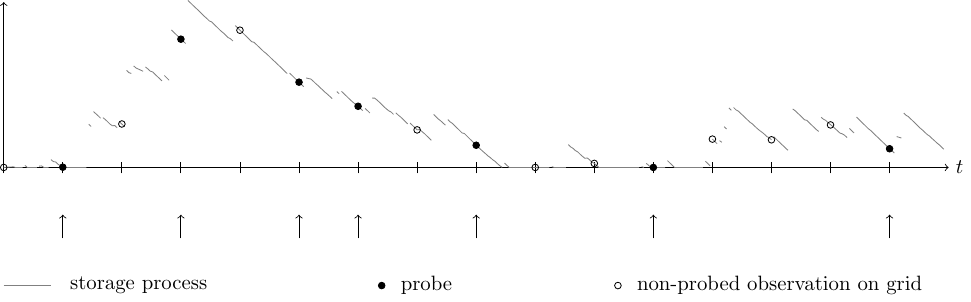}}
	\caption{Discrete probing of a storage process. }\label{f:sampling}%In our model, only the grid observations are observed; the rest of the sample path is merely included for illustration.}
\end{figure*}

\section{Consistency}\label{sec:cons}
In this section, the `almost surely' statements are with respect to any arbitrary distribution for $V(0)$. Importantly, the consistency applies regardless of the value of the grid width $\Delta.$

\begin{theorem}\label{t:cons}
Let $\xi,\Delta > 0$ and $\alpha \geq 0$. Then \begin{equation*}
\lim_{n\to\infty} \hat{\varphi}_n^\Delta(\alpha) = \varphi(\alpha) \text{ a.s.}
\end{equation*}
\end{theorem}

\begin{proof}
The workload process $\{V(t)\}_{t \geq 0}$ is a Harris recurrent Markov process. By Proposition VII.3.8 in \cite{a2003}, the `skeleton process' $\{V(i\Delta)\}_{i \in \bbN}$ is a Harris recurrent Markov chain. By a Markov chain ergodic theorem (see, e.g., Theorem 14.2.11 in \cite{athreya2006}), 
\begin{equation}\label{e:ergodic}
n^{-1} \sum_{i=1}^n f(V(i\Delta)) \to \int_{[0,\infty)} f\,\dd \pi \text{ a.s.}
\end{equation}
for any function $f$ on $[0,\infty)$ integrable with respect to the stationary distribution $\pi$ of $V(\cdot)$. 

The sequence of discrete sampling times is independent of the workload process and therefore satisfies the Weak Lack of Anticipation Assumption required for PASTA (see Theorem 2 of \cite{mmw1989}). 
%It follows from a {\sc pasta}-like argument as in Theorem 2 of \cite{mmw1989} 
In particular, the average at sampling instants converges to the time-average limit and \eqref{e:ergodic} also holds with $V_i^\Delta$ replacing $V(i\Delta)$. Taking $f(v) = \one\{v=0\}$ and $f(v)=e^{-\alpha v}$ yields
\begin{equation*}
\hat{\varphi}_n^\Delta(\alpha) \to \frac{\alpha \pi(\{0\})}{\displaystyle \int_{[0,\infty)} e^{-\alpha v}\,\pi(\dd v)} \text{ a.s.} \end{equation*}
Now the identities \eqref{e:piexp} and \eqref{e:pizero} complete the proof.
\end{proof}

\section{High-frequency central limit theorem}\label{sec:clt}

This section provides a pointwise central limit theorem for the estimation error of $\hat\varphi_n^\Delta$. To do so we consider a high-frequency sampling scheme with grid width $\Delta = \Delta_n = n^{-\gamma}$ for some $\gamma>0$. %so that the total duration of the observation period is $N\Delta=n^{1-\gamma}$. Hence, selecting $\gamma< 1$ ensures that the process is observed for an increasing period of time as $n$ grows.  
%The corresponding asymptotic variance being expressed in terms of the L\'evy exponent at the target argument $\alpha$ as well as $2\alpha$.  
We characterize the processes for which the result holds through the Blumenthal-Getoor index (see  \cite{bg1961})
\[ \beta  = \inf \left\lbrace b > 0 : \int_{(0,1)} x^b \,\nu(\dd x) < \infty \right\rbrace \in [0,1].\]

\begin{theorem}\label{t:clt}
Suppose that $J(\cdot)$ has Blumenthal-Getoor index $\beta<1/4$, and $\Delta = \Delta_n = n^{-\gamma}$ for some $\gamma\in\left( 1/(2-2\sqrt{\beta}),1\right)$.
Assume that the workload process $V(\cdot)$ is stationary. Then as $n\to\infty$,
\begin{equation}\label{e:clt}
\sqrt n\, (\hat\varphi_n^\Delta(\alpha) - \varphi(\alpha)) \rightsquigarrow \mathcal N(0,\sigma^2),
\end{equation}
where $\rightsquigarrow$ denotes weak convergence and
\begin{equation}\label{e:avar}
\sigma^2 = \frac{\varphi(\alpha)^2}{\alpha\varphi'(0)}\Big(\alpha + 2\xi\Big(1-\frac{2\varphi(\alpha)}{\varphi(2\alpha)}\Big) + \frac{2\varphi(\alpha)}{\varphi(2\alpha)}(\varphi(\alpha)-\varphi(2\alpha)) \Big).
\end{equation}
\end{theorem}

\begin{remark}\label{rem:BG}
As $\beta=0$ for compound Poisson input, the result covers the M/G/1 case. The Gamma process is an additional example with $\beta=0$. The Inverse Gaussian process has $\beta=1/2$, hence it does not satisfy the condition of Theorem~\ref{t:clt}.
We conjecture that a CLT can also be constructed if the BG index is higher than $1/4$, but this likely entails a more elaborate asymptotic variance term that depends on the specific sampling scheme. In the sequel we present simulation experiments that display the expected asymptotic performance for the Inverse Gaussian process.
\end{remark}

\begin{proof}
Note that by \eqref{e:PPest}, 
\begin{equation}\label{eq:fr}\sqrt n (\hat\varphi_n^\Delta(\alpha) - \varphi(\alpha)) = \frac{n^{-1/2}\sum_{i=1}^n Z_i^\Delta}{n^{-1}\sum_{i=1}^n e^{-\alpha V_i^\Delta}} \end{equation}
where $ Z_i^\Delta = (\xi - \varphi(\alpha)) e^{-\alpha V_i^\Delta} - \xi e^{-\alpha V_{i-1}^\Delta} + \alpha \one\{V_i^\Delta = 0\}$. 
The proof strategy is to first establish a CLT for the numerator of \eqref{eq:fr}, which is then later used to prove \eqref{e:clt}.
To this end, we approximate $Z_i^\Delta$ by $ Z_i = (\xi - \varphi(\alpha))e^{-\alpha V_i} - \xi e^{-\alpha V_{i-1}} + \alpha \one\{V_i = 0\}$, i.e., the counterpart corresponding to the non-truncated observation times.

The pair $(V_{i-1},V_i)$ is stationary for any $i\geq 1$. As $Z_i$ is a measurable function of $(V_{i-1},V_i)$ it is also a stationary sequence. By the identity \eqref{e:lapVi}, we have $\bbE[Z_i \,|\, V_{i-1}] = 0$, so $Z_i$ is a stationary, ergodic martingale difference sequence. A martingale CLT (see e.g.\ Theorem 18.3 in \cite{b1999}) yields, for $\bbE Z_1^2\in(0,\infty)$, 
\begin{equation}\label{e:numclt}
n^{-1/2}\sum_{i=1}^n Z_i \rightsquigarrow \mathcal N(0,\bbE Z_1^2).
\end{equation}
We show at the end of the proof that
\begin{equation}\label{e:remainder}
n^{-1/2}\sum_{i=1}^n (Z_i^\Delta-Z_i) \rightsquigarrow 0,
\end{equation}
so by Slutsky's Theorem,
\[
n^{-1/2}\sum_{i=1}^n Z_i^\Delta = n^{-1/2}\sum_{i=1}^n Z_i + n^{-1/2}\sum_{i=1}^n (Z_i^\Delta-Z_i) \rightsquigarrow \mathcal N(0,\bbE Z_1^2).
\]
We will also prove below that
\begin{equation}\label{e:erg}
n^{-1} \sum_{i=1}^n e^{-\alpha V_i^\Delta} \rightsquigarrow \frac{\alpha\varphi'(0)}{\varphi(\alpha)}.
\end{equation}
Invoking Slutsky's Theorem again, it follows that 
\begin{equation}\label{e:cltvar}
\sqrt n (\hat\varphi_n^\Delta(\alpha) - \varphi(\alpha)) = \frac{n^{-1/2}\sum_{i=1}^n Z_i^\Delta}{n^{-1}\sum_{i=1}^n e^{-\alpha V_i^\Delta}} \rightsquigarrow \mathcal N(0,\sigma^2), \qquad \sigma^2 = \frac{\varphi(\alpha)^2 \bbE Z_1^2}{(\alpha\varphi'(0))^2}.
\end{equation}

Let us first show that the asymptotic variance appearing in \eqref{e:cltvar} is indeed the expression given in \eqref{e:avar}. Note that by the identity \eqref{e:lapVi},
\begin{align*}
\xi(\xi-\varphi(\alpha)) \,\bbE e^{-\alpha V_{i-1}} e^{-\alpha V_i} 
& = \xi\, \bbE \Big[ e^{-\alpha V_{i-1}} (\xi-\varphi(\alpha)) \bbE[ e^{-\alpha V_i} \,|\, V_{i-1}] \Big] \\
& = \xi\, \bbE \Big[e^{-\alpha V_{i-1}} (\xi e^{-\alpha V_{i-1}} - \alpha \bbP(V_i=0\,|\,V_{i-1}))\Big] \\
& = \xi^2 \,\bbE e^{-2\alpha V_i} - \alpha\xi\, \bbE e^{-\alpha V_{i-1}}\one\{V_i=0\}.
\end{align*}
As a consequence, the variance of $Z_i$ can be computed as
\begin{align}
\notag
\bbE Z_i^2
& = ((\xi-\varphi(\alpha))^2 + \xi^2) \bbE e^{-2 \alpha V_i} + (\alpha^2+2\alpha(\xi-\varphi(\alpha))) \bbP(V_i=0) \\
\notag
&\phantom{={}} - 2\xi(\xi-\varphi(\alpha))\, \bbE e^{-\alpha V_i - \alpha V_{i-1}} - 2\alpha\xi \,\bbE e^{-\alpha V_{i-1}} \one\{V_i=0\}\\
\label{e:cltvar2}
& = \Big((\varphi(\alpha)^2 - 2\xi\varphi(\alpha)) \frac{2}{\varphi(2\alpha)} + \alpha + 2(\xi-\varphi(\alpha))\Big)\alpha\varphi'(0).
\end{align}
The identity for $\sigma^2$ now follows by combining \eqref{e:cltvar} and \eqref{e:cltvar2}.

It remains to be shown that \eqref{e:remainder} and \eqref{e:erg} hold. We prove \eqref{e:remainder} by showing that the term in question has vanishing $L^1$ norm. Recognizing the telescoping sums, and noting that $V_0^\Delta=V_0$, 
\begin{align*}
 \bbE\Big| \sum_{i=1}^n e^{-\alpha V_i^\Delta}-e^{-\alpha V_{i-1}^\Delta}&-e^{-\alpha V_i}+e^{-\alpha V_{i-1}} \Big| = \bbE |e^{-\alpha V_n^\Delta} - e^{-\alpha V_n}| \leq 1.
\end{align*}
as a consequence it follows from the triangle inequality that 
\begin{align*}
\bbE\Big| \sum_{i=1}^n (Z_i^\Delta-Z_i) \Big| &\leq \varphi(\alpha) \sum_{i=1}^n \bbE\big|e^{-\alpha V_i}-e^{-\alpha V_i^\Delta}\big| + \xi\,\bbE\Big| \sum_{i=1}^n e^{-\alpha V_i^\Delta}-e^{-\alpha V_{i-1}^\Delta}-e^{-\alpha V_i}+e^{-\alpha V_{i-1}} \Big| \\
&\phantom{={}} +\alpha \sum_{i=1}^n \bbE\big|\one\{V_i^\Delta = 0\} - \one\{V_i = 0\}\big| \\
&\leq \varphi(\alpha) \sum_{i=1}^n \bbE\big|e^{-\alpha V_i}-e^{-\alpha V_i^\Delta}\big| + \xi+\alpha \sum_{i=1}^n \bbE\big|\one\{V_i^\Delta = 0\} - \one\{V_i = 0\}\big|. \end{align*}

We conclude by showing that both sums in the above display vanish when multiplied by $n^{-1/2}$. Then \eqref{e:remainder} is a direct consequence, and \eqref{e:erg} follows because
\[ \bbE \Big| n^{-1} \sum_{i=1}^n (e^{-\alpha V_i^\Delta}-e^{-\alpha V_i})\Big| \leq n^{-1} \sum_{i=1}^n \bbE |e^{-\alpha V_i^\Delta}-e^{-\alpha V_i}| \to 0, \]
so by Slutsky's Theorem, the proof of Theorem \ref{t:cons}, and the identity \eqref{e:piexp},
\[ n^{-1} \sum_{i=1}^n e^{-\alpha V_i^\Delta} = n^{-1} \sum_{i=1}^n (e^{-\alpha V_i^\Delta}-e^{-\alpha V_i}) + n^{-1} \sum_{i=1}^n e^{-\alpha V_i} \rightsquigarrow \int e^{-\alpha v} \,\pi(\dd v) = \frac{\alpha\varphi'(0)}{\varphi(\alpha)}. \]
Here we consider separately the cases of finite and infinite intensity.

\noindent \textit{$\circ$~Finite intensity case.} By conditioning on the number of jumps between $S_i$ and $S_i^\Delta$,it follows that
\[ \bbE|e^{-\alpha V_i} - e^{-\alpha V_i^\Delta}| \leq \int (e^{-\alpha(x-\Delta/2)}-e^{-\alpha x}) \,\pi(\dd x) + O(\Delta) = O(\Delta), \]
using $|S_i-S_i^\Delta| \leq \Delta/2$. 
Regarding the other term, note that
\[ \bbE|\one\{V_i^\Delta = 0\} - \one\{V_i = 0\}| = \bbP(V_i^\Delta = 0) + \bbP(V_i = 0) - 2\, \bbP(V_i^\Delta = V_i = 0).  \]
By the stationarity of $V(\cdot)$, 
\begin{align*}
\bbP(V_i^\Delta = 0) - \bbP(V_i^\Delta = V_i = 0)
&= (1 - \bbP_0(V(|S_i^\Delta - S_i|)=0)) \pi(\{0\}) \\
& \leq \bbP(\text{at least one arrival between } S_i \text{ and }S_i^\Delta)  = O(\Delta).
\end{align*}
Both $O(\Delta)$ terms can be bounded by $C\Delta$ for some $C>0$ independent of $i, n$. This means that the $L^1$ norm of the remainder term, $ n^{-1/2} \bbE\,\big| \sum_{i=1}^n (Z_i^\Delta-Z_i) \big|$,
can be bounded by a constant times $\Delta \sqrt{n}$. Recalling the assumption {that $\Delta=n^{-\gamma}$ for $\gamma>1/2$}, conclude that this product vanishes.

\noindent \textit{$\circ$~Infinite intensity case.} Let $\beta_0 = (\frac{2\gamma-1}{2\gamma})^2$, so that $\beta < \beta_0 < 1$, which follows from the assumption on $\gamma$. The following fact, established in the proof of Theorem 3.1 in \cite{bg1961}, will be used later on: given $\delta > \epsilon > \beta$, there exists a positive constant $C$ depending on $\delta$ and $\epsilon$ such that for all sufficiently small $s>0$,
\begin{equation}\label{e:bg2}
\bbP(J(s) > s^{1/\delta}) \leq C s^{1-\epsilon/\delta}.
\end{equation}

Inspecting the proof of the finite-intensity case, it suffices to show that the quantities
\[ A(\Delta) := \bbE|e^{-\alpha V_i}-e^{-\alpha V_i^\Delta}|, \qquad B(\Delta) := 1 - \bbP_0(V(|S_i^\Delta - S_i|) = 0)\]
are both $o(n^{-1/2})$. Applying \eqref{e:bg2} with $\delta = 1$ and $\epsilon=\beta_0$, we obtain 
\begin{equation}\label{e:bigjumps}
\bbP(J(s)>s) \leq C s^{1-\beta_0}
\end{equation}
so, by separately considering the event $\{ V(s) > x \}$ and its complement, we find
\begin{align}
\nonumber
\bbE_x|e^{-\alpha V(s)}-e^{-\alpha x}|
& \leq \bbP_x(V(s)>x) + \bbE_x(e^{-\alpha V(s)}-e^{-\alpha x}) \boldsymbol 1\{V(s)\leq x\} \\
\nonumber
& \leq \bbP(J(s)>s) + \bbE_x \alpha (x-V(s)) \one\{V(s) \leq x\} 
\label{e:expdb}
 \leq Cs^{1-\beta_0} + \alpha s,
\end{align}
as $s\to 0$, where in the bottom line we use that $V(s) \geq x-s$ given $V(0)=x$. It follows that
\[ \int \bbE_x|e^{-\alpha V(s)} - e^{-\alpha x}| \,\pi(\dd x) \leq Cs^{1-\beta_0} + \alpha s \leq (C+\alpha)s^{1-\beta_0} \]
for sufficiently small $s>0$. By conditioning on $\min\{S_i,S_i^\Delta\}$ and using that $V(\cdot)$ is a stationary Markov process, it now follows that
\begin{align*}
A(\Delta)
&    = \int \bbE_x|e^{-\alpha V(|S_i^\Delta-S_i|)} - e^{-\alpha x}| \,\pi(\dd x) \leq (C+\alpha) \bbE |S_i-S_i^\Delta|^{1-\beta_0} \\
& \leq (C+\alpha) \Delta^{1-\beta_0}    = (C+\alpha) n^{-(4\gamma-1)/(4\gamma)} = o(n^{-1/2}).
\end{align*}

Regarding the term $B(\Delta)$, Theorem 2 in \cite{t1966} gives, for $s>0$,
\begin{align}
\nonumber
1 - \bbP_0(V(s) = 0)
& = 1 - \int_0^s (1-u/s) \,\bbP(J(s) \in \dd u) = \bbP(J(s)>s) + s^{-1} \int_0^s u \, \bbP(J(s) \in \dd u) \\
\label{e:p0db}
& \leq C s^{1-\beta_0} + s^{-1} \int_0^s u \, \bbP(J(s) \in \dd u),
\end{align}
where the last line follows from \eqref{e:bigjumps}. By \eqref{e:bg2} with $\delta = \beta_0$ and $\epsilon = \beta_0^2$, we thus obtain the bound
\begin{align*}
s^{-1} \int_0^s u \,\bbP(J(s) \in \dd u)
&    = \bbE s^{-1} J(s) \boldsymbol 1 \{J(s) \leq s^{1/\beta_0}\} + \bbE s^{-1} J(s) \boldsymbol 1 \{s^{1/\beta_0} < J(s) \leq s\} \\
& \leq s^{1/\beta_0-1} + \bbP(J(s) > s^{1/\beta_0})\leq s^{1/\beta_0-1} + C s^{1-\beta_0}
\end{align*}
(possibly for different $C$ than above) for all $s$ sufficiently small. Since $1-\beta_0 < 1/\beta_0-1$, it follows that the dominating term in \eqref{e:p0db} is of the order $s^{1-\beta_0}$, so $B(\Delta)$ is at most of the order $\Delta^{1-\beta_0}$, which is $o(n^{-1/2})$ as we already showed.
\end{proof}

\section{Simulation analysis}\label{sec:sim}

\begin{figure}[h]\begin{center}
    \begin{minipage}{0.47\textwidth}
        \begin{tikzpicture}[scale=.7]
\draw[<->] (0, 6) -- (0, 0) -- (10, 0) node[right]{$\alpha$};

\foreach \x in {2, 4, 6, 8}
{
\draw (\x, .05) -- (\x, -.05) node[below]{$\x$};
}
\foreach \y in {2, 4}
{
\draw (.05, \y) -- (-.05, \y) node[left]{$\y$};
}

\draw (1, 4.25) -- (3.5, 4.25) -- (3.5, 6) -- (1, 6) -- (1, 4.25);

\draw[black!50] (1.5, 5.5) -- (2.5, 5.5) node[black, right]{$\hat \varphi_n^\Delta$};
\draw[dashed, black!50] (1.5, 5.7) -- (2.5, 5.7);
\draw[dashed, black!50] (1.5, 5.3) -- (2.5, 5.3);
\draw           (1.5, 4.75) -- (2.5, 4.75) node[black, right]{$\varphi_\epsilon$};

% actual
\draw (0, 0) -- (0.1, 0.021) -- (0.2, 0.043) -- (0.3, 0.067) -- (0.4, 0.092) -- (0.5, 0.118) -- (0.6, 0.145) -- (0.7, 0.174) -- (0.8, 0.203) -- (0.9, 0.234) -- (1, 0.266) -- (1.1, 0.298) -- (1.2, 0.332) -- (1.3, 0.366) -- (1.4, 0.402) -- (1.5, 0.438) -- (1.6, 0.475) -- (1.7, 0.513) -- (1.8, 0.551) -- (1.9, 0.59) -- (2, 0.631) -- (2.1, 0.671) -- (2.2, 0.713) -- (2.3, 0.755) -- (2.4, 0.798) -- (2.5, 0.841) -- (2.6, 0.885) -- (2.7, 0.93) -- (2.8, 0.975) -- (2.9, 1.021) -- (3, 1.068) -- (3.1, 1.115) -- (3.2, 1.162) -- (3.3, 1.21) -- (3.4, 1.259) -- (3.5, 1.308) -- (3.6, 1.357) -- (3.7, 1.407) -- (3.8, 1.457) -- (3.9, 1.508) -- (4, 1.56) -- (4.1, 1.611) -- (4.2, 1.664) -- (4.3, 1.716) -- (4.4, 1.769) -- (4.5, 1.823) -- (4.6, 1.876) -- (4.7, 1.931) -- (4.8, 1.985) -- (4.9, 2.04) -- (5, 2.095) -- (5.1, 2.151) -- (5.2, 2.207) -- (5.3, 2.263) -- (5.4, 2.32) -- (5.5, 2.377) -- (5.6, 2.434) -- (5.7, 2.492) -- (5.8, 2.55) -- (5.9, 2.608) -- (6, 2.666) -- (6.1, 2.725) -- (6.2, 2.784) -- (6.3, 2.844) -- (6.4, 2.903) -- (6.5, 2.963) -- (6.6, 3.023) -- (6.7, 3.084) -- (6.8, 3.145) -- (6.9, 3.206) -- (7, 3.267) -- (7.1, 3.328) -- (7.2, 3.39) -- (7.3, 3.452) -- (7.4, 3.514) -- (7.5, 3.577) -- (7.6, 3.639) -- (7.7, 3.702) -- (7.8, 3.765) -- (7.9, 3.829) -- (8, 3.892) -- (8.1, 3.956) -- (8.2, 4.02) -- (8.3, 4.084) -- (8.4, 4.149) -- (8.5, 4.213) -- (8.6, 4.278) -- (8.7, 4.343) -- (8.8, 4.408) -- (8.9, 4.474) -- (9, 4.539) -- (9.1, 4.605) -- (9.2, 4.671) -- (9.3, 4.737) -- (9.4, 4.803) -- (9.5, 4.87) -- (9.6, 4.936) -- (9.7, 5.003) -- (9.8, 5.07) -- (9.9, 5.137) -- (10, 5.205);

% estimator
\draw[black!50] (0, 0.001) -- (0.1, 0.023) -- (0.2, 0.046) -- (0.3, 0.071) -- (0.4, 0.096) -- (0.5, 0.123) -- (0.6, 0.151) -- (0.7, 0.18) -- (0.8, 0.21) -- (0.9, 0.242) -- (1, 0.274) -- (1.1, 0.308) -- (1.2, 0.343) -- (1.3, 0.379) -- (1.4, 0.416) -- (1.5, 0.453) -- (1.6, 0.492) -- (1.7, 0.532) -- (1.8, 0.573) -- (1.9, 0.615) -- (2, 0.657) -- (2.1, 0.701) -- (2.2, 0.745) -- (2.3, 0.791) -- (2.4, 0.837) -- (2.5, 0.884) -- (2.6, 0.931) -- (2.7, 0.98) -- (2.8, 1.029) -- (2.9, 1.079) -- (3, 1.13) -- (3.1, 1.182) -- (3.2, 1.234) -- (3.3, 1.287) -- (3.4, 1.341) -- (3.5, 1.395) -- (3.6, 1.45) -- (3.7, 1.506) -- (3.8, 1.562) -- (3.9, 1.619) -- (4, 1.676) -- (4.1, 1.734) -- (4.2, 1.793) -- (4.3, 1.852) -- (4.4, 1.912) -- (4.5, 1.972) -- (4.6, 2.033) -- (4.7, 2.094) -- (4.8, 2.156) -- (4.9, 2.218) -- (5, 2.281) -- (5.1, 2.344) -- (5.2, 2.408) -- (5.3, 2.472) -- (5.4, 2.537) -- (5.5, 2.602) -- (5.6, 2.668) -- (5.7, 2.734) -- (5.8, 2.8) -- (5.9, 2.867) -- (6, 2.934) -- (6.1, 3.002) -- (6.2, 3.069) -- (6.3, 3.138) -- (6.4, 3.206) -- (6.5, 3.275) -- (6.6, 3.345) -- (6.7, 3.414) -- (6.8, 3.484) -- (6.9, 3.555) -- (7, 3.625) -- (7.1, 3.696) -- (7.2, 3.768) -- (7.3, 3.839) -- (7.4, 3.911) -- (7.5, 3.983) -- (7.6, 4.056) -- (7.7, 4.129) -- (7.8, 4.202) -- (7.9, 4.275) -- (8, 4.348) -- (8.1, 4.422) -- (8.2, 4.496) -- (8.3, 4.571) -- (8.4, 4.645) -- (8.5, 4.72) -- (8.6, 4.795) -- (8.7, 4.87) -- (8.8, 4.946) -- (8.9, 5.021) -- (9, 5.097) -- (9.1, 5.173) -- (9.2, 5.25) -- (9.3, 5.326) -- (9.4, 5.403) -- (9.5, 5.48) -- (9.6, 5.557) -- (9.7, 5.634) -- (9.8, 5.712) -- (9.9, 5.789) -- (10, 5.867);

% upper
\draw[dashed, black!50] (0, 0.001) -- (0.1, 0.032) -- (0.2, 0.067) -- (0.3, 0.102) -- (0.4, 0.139) -- (0.5, 0.178) -- (0.6, 0.217) -- (0.7, 0.258) -- (0.8, 0.301) -- (0.9, 0.344) -- (1, 0.389) -- (1.1, 0.435) -- (1.2, 0.483) -- (1.3, 0.532) -- (1.4, 0.581) -- (1.5, 0.633) -- (1.6, 0.685) -- (1.7, 0.738) -- (1.8, 0.792) -- (1.9, 0.848) -- (2, 0.904) -- (2.1, 0.962) -- (2.2, 1.02) -- (2.3, 1.08) -- (2.4, 1.14) -- (2.5, 1.201) -- (2.6, 1.263) -- (2.7, 1.326) -- (2.8, 1.39) -- (2.9, 1.455) -- (3, 1.521) -- (3.1, 1.587) -- (3.2, 1.654) -- (3.3, 1.722) -- (3.4, 1.79) -- (3.5, 1.859) -- (3.6, 1.929) -- (3.7, 2) -- (3.8, 2.071) -- (3.9, 2.143) -- (4, 2.216) -- (4.1, 2.289) -- (4.2, 2.363) -- (4.3, 2.437) -- (4.4, 2.512) -- (4.5, 2.587) -- (4.6, 2.663) -- (4.7, 2.74) -- (4.8, 2.817) -- (4.9, 2.895) -- (5, 2.973) -- (5.1, 3.051) -- (5.2, 3.13) -- (5.3, 3.21) -- (5.4, 3.289) -- (5.5, 3.37) -- (5.6, 3.45) -- (5.7, 3.532) -- (5.8, 3.613) -- (5.9, 3.695) -- (6, 3.777) -- (6.1, 3.86) -- (6.2, 3.943) -- (6.3, 4.027) -- (6.4, 4.11) -- (6.5, 4.195) -- (6.6, 4.279) -- (6.7, 4.364) -- (6.8, 4.449) -- (6.9, 4.534) -- (7, 4.62) -- (7.1, 4.706) -- (7.2, 4.792) -- (7.3, 4.879) -- (7.4, 4.966) -- (7.5, 5.053) -- (7.6, 5.14) -- (7.7, 5.228) -- (7.8, 5.316) -- (7.9, 5.404) -- (8, 5.492) -- (8.1, 5.581) -- (8.2, 5.669) -- (8.3, 5.758) -- (8.4, 5.848) -- (8.5, 5.937) -- (8.6, 6.027) -- (8.7, 6.117) -- (8.8, 6.207) -- (8.9, 6.297) -- (9, 6.388) -- (9.1, 6.478) -- (9.2, 6.569) -- (9.3, 6.66) -- (9.4, 6.752) -- (9.5, 6.843) -- (9.6, 6.935) -- (9.7, 7.026) -- (9.8, 7.118) -- (9.9, 7.21) -- (10, 7.303);
%
% lower
\draw[dashed, black!50] (0, 0.001) -- (0.1, 0.015) -- (0.2, 0.026) -- (0.3, 0.039) -- (0.4, 0.053) -- (0.5, 0.068) -- (0.6, 0.085) -- (0.7, 0.102) -- (0.8, 0.12) -- (0.9, 0.14) -- (1, 0.16) -- (1.1, 0.181) -- (1.2, 0.203) -- (1.3, 0.226) -- (1.4, 0.25) -- (1.5, 0.274) -- (1.6, 0.3) -- (1.7, 0.326) -- (1.8, 0.354) -- (1.9, 0.382) -- (2, 0.41) -- (2.1, 0.44) -- (2.2, 0.47) -- (2.3, 0.501) -- (2.4, 0.533) -- (2.5, 0.566) -- (2.6, 0.599) -- (2.7, 0.633) -- (2.8, 0.668) -- (2.9, 0.704) -- (3, 0.74) -- (3.1, 0.777) -- (3.2, 0.814) -- (3.3, 0.852) -- (3.4, 0.891) -- (3.5, 0.93) -- (3.6, 0.97) -- (3.7, 1.011) -- (3.8, 1.052) -- (3.9, 1.094) -- (4, 1.136) -- (4.1, 1.179) -- (4.2, 1.223) -- (4.3, 1.267) -- (4.4, 1.311) -- (4.5, 1.357) -- (4.6, 1.402) -- (4.7, 1.448) -- (4.8, 1.495) -- (4.9, 1.542) -- (5, 1.59) -- (5.1, 1.638) -- (5.2, 1.686) -- (5.3, 1.735) -- (5.4, 1.785) -- (5.5, 1.835) -- (5.6, 1.885) -- (5.7, 1.936) -- (5.8, 1.987) -- (5.9, 2.039) -- (6, 2.091) -- (6.1, 2.143) -- (6.2, 2.196) -- (6.3, 2.249) -- (6.4, 2.302) -- (6.5, 2.356) -- (6.6, 2.411) -- (6.7, 2.465) -- (6.8, 2.52) -- (6.9, 2.575) -- (7, 2.631) -- (7.1, 2.687) -- (7.2, 2.743) -- (7.3, 2.8) -- (7.4, 2.857) -- (7.5, 2.914) -- (7.6, 2.972) -- (7.7, 3.03) -- (7.8, 3.088) -- (7.9, 3.146) -- (8, 3.205) -- (8.1, 3.264) -- (8.2, 3.323) -- (8.3, 3.383) -- (8.4, 3.442) -- (8.5, 3.502) -- (8.6, 3.563) -- (8.7, 3.623) -- (8.8, 3.684) -- (8.9, 3.745) -- (9, 3.806) -- (9.1, 3.868) -- (9.2, 3.93) -- (9.3, 3.992) -- (9.4, 4.054) -- (9.5, 4.116) -- (9.6, 4.179) -- (9.7, 4.242) -- (9.8, 4.305) -- (9.9, 4.368) -- (10, 4.431);
\end{tikzpicture}
        \caption{Approximate pointwise confidence intervals, using $\xi = 1$, $\Delta = 4 \cdot 10^{-4} \approx 100^{-1.707}$, $N\Delta=100$. In this simulation we obtained $n=105$.}
        \label{f:confInt}
    \end{minipage}
    \hfill
    \begin{minipage}{0.47\textwidth}
        \begin{tikzpicture}
\foreach \x in {(-2.807, 2.273), (-2.432, 2.412), (-2.241, 2.548), (-2.108, 2.584), (-2.005, 2.596), (-1.919, 2.709), (-1.845, 2.758), (-1.78, 2.779), (-1.722, 2.79), (-1.67, 2.796), (-1.621, 2.85), (-1.576, 2.958), (-1.534, 3.055), (-1.495, 3.057), (-1.457, 3.06), (-1.422, 3.075), (-1.388, 3.115), (-1.356, 3.143), (-1.326, 3.144), (-1.296, 3.151), (-1.267, 3.175), (-1.24, 3.186), (-1.213, 3.221), (-1.188, 3.251), (-1.163, 3.31), (-1.138, 3.356), (-1.115, 3.36), (-1.092, 3.37), (-1.069, 3.386), (-1.047, 3.423), (-1.026, 3.439), (-1.005, 3.439), (-0.984, 3.444), (-0.964, 3.471), (-0.944, 3.478), (-0.925, 3.484), (-0.906, 3.491), (-0.887, 3.511), (-0.869, 3.531), (-0.851, 3.542), (-0.833, 3.583), (-0.815, 3.592), (-0.798, 3.595), (-0.781, 3.601), (-0.764, 3.604), (-0.747, 3.622), (-0.731, 3.652), (-0.714, 3.677), (-0.698, 3.688), (-0.682, 3.703), (-0.667, 3.709), (-0.651, 3.711), (-0.636, 3.722), (-0.62, 3.729), (-0.605, 3.739), (-0.59, 3.755), (-0.575, 3.759), (-0.561, 3.762), (-0.546, 3.762), (-0.532, 3.775), (-0.517, 3.782), (-0.503, 3.808), (-0.489, 3.808), (-0.475, 3.816), (-0.461, 3.823), (-0.447, 3.825), (-0.433, 3.828), (-0.419, 3.839), (-0.406, 3.842), (-0.392, 3.867), (-0.379, 3.88), (-0.365, 3.884), (-0.352, 3.896), (-0.338, 3.897), (-0.325, 3.906), (-0.312, 3.912), (-0.299, 3.935), (-0.286, 3.948), (-0.273, 3.953), (-0.26, 3.953), (-0.247, 3.959), (-0.234, 3.961), (-0.221, 3.987), (-0.208, 4.026), (-0.196, 4.04), (-0.183, 4.04), (-0.17, 4.04), (-0.157, 4.042), (-0.145, 4.043), (-0.132, 4.045), (-0.119, 4.045), (-0.107, 4.051), (-0.094, 4.053), (-0.082, 4.075), (-0.069, 4.081), (-0.056, 4.084), (-0.044, 4.085), (-0.031, 4.087), (-0.019, 4.1), (-0.006, 4.111), (0.006, 4.158), (0.019, 4.161), (0.031, 4.164), (0.044, 4.173), (0.056, 4.174), (0.069, 4.188), (0.082, 4.189), (0.094, 4.196), (0.107, 4.208), (0.119, 4.218), (0.132, 4.219), (0.145, 4.229), (0.157, 4.248), (0.17, 4.26), (0.183, 4.266), (0.196, 4.276), (0.208, 4.284), (0.221, 4.3), (0.234, 4.303), (0.247, 4.314), (0.26, 4.333), (0.273, 4.336), (0.286, 4.336), (0.299, 4.349), (0.312, 4.358), (0.325, 4.391), (0.338, 4.399), (0.352, 4.403), (0.365, 4.405), (0.379, 4.427), (0.392, 4.437), (0.406, 4.443), (0.419, 4.443), (0.433, 4.443), (0.447, 4.457), (0.461, 4.463), (0.475, 4.467), (0.489, 4.504), (0.503, 4.519), (0.517, 4.522), (0.532, 4.551), (0.546, 4.552), (0.561, 4.553), (0.575, 4.555), (0.59, 4.557), (0.605, 4.576), (0.62, 4.591), (0.636, 4.601), (0.651, 4.609), (0.667, 4.616), (0.682, 4.64), (0.698, 4.647), (0.714, 4.647), (0.731, 4.666), (0.747, 4.668), (0.764, 4.673), (0.781, 4.673), (0.798, 4.685), (0.815, 4.699), (0.833, 4.702), (0.851, 4.728), (0.869, 4.736), (0.887, 4.764), (0.906, 4.815), (0.925, 4.816), (0.944, 4.84), (0.964, 4.842), (0.984, 4.851), (1.005, 4.861), (1.026, 4.902), (1.047, 4.908), (1.069, 4.911), (1.092, 4.927), (1.115, 4.938), (1.138, 4.948), (1.163, 4.981), (1.188, 5.003), (1.213, 5.039), (1.24, 5.063), (1.267, 5.132), (1.296, 5.146), (1.326, 5.157), (1.356, 5.174), (1.388, 5.194), (1.422, 5.233), (1.457, 5.236), (1.495, 5.284), (1.534, 5.311), (1.576, 5.315), (1.621, 5.318), (1.67, 5.378), (1.722, 5.416), (1.78, 5.445), (1.845, 5.501), (1.919, 5.511), (2.005, 5.531), (2.108, 5.537), (2.241, 5.635), (2.432, 5.747), (2.807, 6.259)}{
\draw [fill] \x circle [radius=.4pt];
}

\draw[<->] (-3, 6.5) -- (-3, 2) -- (3, 2);

\foreach \x in {-2, -1, 0, 1, 2}
{
\draw (\x, 2.04) -- (\x, 1.96) node[below]{$\x$};
}

\draw (-2.96, 3) -- (-3.04, 3) node[left]{$1.5$};
\draw (-2.96, 4) -- (-3.04, 4) node[left]{$2$};
\draw (-2.96, 5) -- (-3.04, 5) node[left]{$2.5$};
\draw (-2.96, 6) -- (-3.04, 6) node[left]{$3$};

\end{tikzpicture}
        \caption{Normal QQ-plot of 200 realizations of $\hat\varphi_n^\Delta(5)$, with the same parameters as in Figure~\ref{f:confInt}}
        \label{f:qqplot}
    \end{minipage}
\end{center}\end{figure}

To demonstrate the performance of the method we simulate an input process $J(\cdot)$ that is the sum of two independent processes:  $J_1(\cdot)$ is a Gamma process with shape parameter $\gamma$ and rate parameter $\eta$, %, meaning that its Lévy exponent is
%\[ \log \bbE e^{-\alpha J_1(1)} = \int_0^\infty (e^{-\alpha x}-1) \gamma x^{-1} e^{-\beta x} \,\dd x = \gamma \log \frac{\beta}{\beta+\alpha}. \]
%It follows that $J_1(t)$ has a Gamma$(\gamma t,\beta)$ distribution. 
and $J_2(\cdot)$ is an inverse Gaussian process, given by
$ \inf\{s \geq 0 : \lambda^{-1/2} W(s) + \mu^{-1} s = t \}$,
for $W(\cdot)$ a standard Brownian motion and $\lambda,\mu > 0$ positive constants. %Its Lévy exponent is
%\[ \log \bbE e^{-\alpha J_2(1)} = \int_{(0,\infty)} (e^{-\alpha x}-1) \sqrt{\frac{\lambda}{2\pi}} x^{-3/2} e^{-\lambda x/(2\mu^2)} \, \dd x = \frac\lambda\mu \left(1 - \sqrt{1+2\mu^2\alpha/\lambda}\right). \]
The input process $J(t)=J_1(t)+J_2(t)$ has Lévy density $\gamma x^{-1} e^{-\eta x} + \sqrt{\frac{\lambda}{2\pi}} x^{-3/2} e^{-\lambda x/(2\mu^2)}. $ The stability condition is $\gamma/\eta + \mu = \bbE J(1) < 1$, which is fulfilled by our choice $(\gamma,\eta,\mu,\lambda) = (2,5,1,2/5)$.   In the simulation we approximate the Lévy exponent of $J(\cdot)$ by $\int_{(\epsilon,\infty)} (e^{-\alpha x}-1) \,\nu(\dd x)$ for some small $\epsilon$, yielding $ J_\epsilon(t) = \sum_{i=1}^{N(t)} B_i$,  where $N(\cdot)$ is a Poisson process of rate $r_\epsilon = \nu(\epsilon,\infty) < \infty$ and the $B_i$ are i.i.d.\ with distribution $\bbP(B_i \leq x) = r_\epsilon^{-1} \nu(\epsilon,x]$.  In the examples below we take $\epsilon = 10^{-5}$.  %We compute numerically the drift of the simulated process
%\[ \bbE J_\epsilon(1) - 1 = \int_{(\epsilon,\infty)} x \,\nu(\dd x) - 1 \approx - 0.202543 \]
%which is a reasonable approximation to the true drift $-0.2$. 
The simulated workload observations are $\{V(0),V(\Delta),\ldots,V(N\Delta)\}$ for some fixed horizon $N\Delta$. The Lévy exponents $\varphi$ and $\varphi_\epsilon$ are close, as is illustrated in Figure~\ref{f:nonBootstrap}.

Figure~\ref{f:confInt} shows a realisation of the estimator $\hat\varphi_n^\Delta$ and an estimated confidence interval based on Theorem~\ref{t:clt}. The confidence interval is $\hat\varphi_n^\Delta(\alpha) + 1.96 \hat\sigma(\alpha)/\sqrt n$, where $\hat\sigma(\alpha)$ estimates the variance in \eqref{e:avar} by replacing $\varphi$ by $\hat\varphi_n^\Delta$ and $\varphi'(0) = \pi(\{0\})$ by $n^{-1}\sum_{i=1}^n \one \{ V_i^\Delta = 0 \}$. From the figure it can be concluded that this is a reasonable estimate of the variance.

To obtain Figure~\ref{f:qqplot} we simulated 200 workload processes and computed each time the estimator $\hat\varphi_n^\Delta(\alpha)$ at $\alpha = 5$. The sample quantiles from the sequence of realised estimations align well with those of the normal distribution, which corroborates the Central Limit Theorem we have established.

Recall that in our approach the workload process is sampled at a discrete collection of time instants, but only a (random) subset of these observations are effectively used in the estimator. We point out how one can make more efficient use of the available information, an issue that is particularly relevant if the sample size is relatively small. Therefore, to reduce variance due to the stochastic probing, we now consider the natural assumption that observations of $V(\cdot)$ for the entire grid $\{0,\Delta,\ldots,N\Delta\}$ are available, and propose a `resampling estimator' which is the average of $K \in \mathbb N$ realizations $\hat\varphi_{k, n(k)}^\Delta$ of the original estimator:
\[ \hat \varphi_{\text r, K}^\Delta := \frac 1K \sum_{k = 1}^K \hat\varphi_{k, n(k)}^\Delta, \]
where $n(k)$ is the stochastic number of probes in iteration $k$. Recall that for a given grid the number of observations sampled for the approximate estimator is a random variable, hence it may vary between iterations. Taking $K$ large enough  `exhausts the information' in the sample of workload observations on the grid, and thus reduces the aforementioned variability. In practice, $K$ can be chosen empirically, since $\hat \varphi_{\text r, K}^\Delta$ can be computed fast.% (given a sample, one computes the estimator a few times and checks if it does not vary too much).

Figure \ref{f:nonBootstrap} illustrates five simulation realizations of the estimated L\'evy exponent function without resampling, along with the true   $\varphi$ and its approximation  $\varphi_\epsilon$  described above. The approximation performs well for $\alpha \in [0,10]$ (although the exponents grow apart as $\alpha\to\infty$).  The reflected process associated with $X_\epsilon(t) = J_\epsilon(t) - t$ has been simulated for $t \in [0,25]$ and the estimator $\hat\varphi_n^\Delta$ has been evaluated for every realization of the process. The main conclusion from the plot is that the sampling scheme causes much variability in the estimator, when sample sizes are small.  %This is in line with \cite{r2023} that showed that the asymptotic variance of the estimation error grows quadratically with $\alpha$.
\begin{figure}\begin{center}
\begin{minipage}{0.47\textwidth}
\centering
\begin{tikzpicture}[scale=.66]
\draw[<->] (0, 6) -- (0, 0) -- (10, 0) node[right]{$\alpha$};

\foreach \x in {2, 4, 6, 8}
{
\draw (\x, .05) -- (\x, -.05) node[below]{$\x$};
}
\foreach \y in {2, 4}
{
\draw (.05, \y) -- (-.05, \y) node[left]{$\y$};
}

\draw (1, 3.5) -- (3.75, 3.5) -- (3.75, 6) -- (1, 6) -- (1, 3.5);

\draw[black!50] (1.5, 5.5) -- (2.5, 5.5) node[black, right]{$\hat\varphi_n^\Delta$};
\draw           (1.5, 4.75) -- (2.5, 4.75) node[black, right]{$\varphi_\epsilon$};
\draw[dashed]   (1.5, 4) -- (2.5, 4) node[black, right]{$\varphi$};

% true
\draw[dashed] (0, 0) -- (0.1, 0.021) -- (0.2, 0.043) -- (0.3, 0.066) -- (0.4, 0.091) -- (0.5, 0.117) -- (0.6, 0.144) -- (0.7, 0.172) -- (0.8, 0.201) -- (0.9, 0.232) -- (1, 0.263) -- (1.1, 0.295) -- (1.2, 0.329) -- (1.3, 0.363) -- (1.4, 0.398) -- (1.5, 0.434) -- (1.6, 0.471) -- (1.7, 0.508) -- (1.8, 0.547) -- (1.9, 0.586) -- (2, 0.625) -- (2.1, 0.666) -- (2.2, 0.707) -- (2.3, 0.749) -- (2.4, 0.792) -- (2.5, 0.835) -- (2.6, 0.879) -- (2.7, 0.923) -- (2.8, 0.968) -- (2.9, 1.014) -- (3, 1.06) -- (3.1, 1.107) -- (3.2, 1.154) -- (3.3, 1.202) -- (3.4, 1.25) -- (3.5, 1.299) -- (3.6, 1.348) -- (3.7, 1.398) -- (3.8, 1.448) -- (3.9, 1.498) -- (4, 1.55) -- (4.1, 1.601) -- (4.2, 1.653) -- (4.3, 1.705) -- (4.4, 1.758) -- (4.5, 1.811) -- (4.6, 1.865) -- (4.7, 1.919) -- (4.8, 1.973) -- (4.9, 2.028) -- (5, 2.083) -- (5.1, 2.138) -- (5.2, 2.194) -- (5.3, 2.25) -- (5.4, 2.306) -- (5.5, 2.363) -- (5.6, 2.42) -- (5.7, 2.477) -- (5.8, 2.535) -- (5.9, 2.593) -- (6, 2.651) -- (6.1, 2.71) -- (6.2, 2.768) -- (6.3, 2.828) -- (6.4, 2.887) -- (6.5, 2.947) -- (6.6, 3.007) -- (6.7, 3.067) -- (6.8, 3.127) -- (6.9, 3.188) -- (7, 3.249) -- (7.1, 3.31) -- (7.2, 3.372) -- (7.3, 3.433) -- (7.4, 3.495) -- (7.5, 3.558) -- (7.6, 3.62) -- (7.7, 3.683) -- (7.8, 3.746) -- (7.9, 3.809) -- (8, 3.872) -- (8.1, 3.936) -- (8.2, 3.999) -- (8.3, 4.063) -- (8.4, 4.127) -- (8.5, 4.192) -- (8.6, 4.256) -- (8.7, 4.321) -- (8.8, 4.386) -- (8.9, 4.451) -- (9, 4.516) -- (9.1, 4.582) -- (9.2, 4.648) -- (9.3, 4.713) -- (9.4, 4.779) -- (9.5, 4.846) -- (9.6, 4.912) -- (9.7, 4.979) -- (9.8, 5.045) -- (9.9, 5.112) -- (10, 5.179);

% actual
\draw (0, 0) -- (0.1, 0.021) -- (0.2, 0.043) -- (0.3, 0.067) -- (0.4, 0.092) -- (0.5, 0.118) -- (0.6, 0.145) -- (0.7, 0.174) -- (0.8, 0.203) -- (0.9, 0.234) -- (1, 0.266) -- (1.1, 0.298) -- (1.2, 0.332) -- (1.3, 0.366) -- (1.4, 0.402) -- (1.5, 0.438) -- (1.6, 0.475) -- (1.7, 0.513) -- (1.8, 0.551) -- (1.9, 0.59) -- (2, 0.631) -- (2.1, 0.671) -- (2.2, 0.713) -- (2.3, 0.755) -- (2.4, 0.798) -- (2.5, 0.841) -- (2.6, 0.885) -- (2.7, 0.93) -- (2.8, 0.975) -- (2.9, 1.021) -- (3, 1.068) -- (3.1, 1.115) -- (3.2, 1.162) -- (3.3, 1.21) -- (3.4, 1.259) -- (3.5, 1.308) -- (3.6, 1.357) -- (3.7, 1.407) -- (3.8, 1.457) -- (3.9, 1.508) -- (4, 1.56) -- (4.1, 1.611) -- (4.2, 1.664) -- (4.3, 1.716) -- (4.4, 1.769) -- (4.5, 1.823) -- (4.6, 1.876) -- (4.7, 1.931) -- (4.8, 1.985) -- (4.9, 2.04) -- (5, 2.095) -- (5.1, 2.151) -- (5.2, 2.207) -- (5.3, 2.263) -- (5.4, 2.32) -- (5.5, 2.377) -- (5.6, 2.434) -- (5.7, 2.492) -- (5.8, 2.55) -- (5.9, 2.608) -- (6, 2.666) -- (6.1, 2.725) -- (6.2, 2.784) -- (6.3, 2.844) -- (6.4, 2.903) -- (6.5, 2.963) -- (6.6, 3.023) -- (6.7, 3.084) -- (6.8, 3.145) -- (6.9, 3.206) -- (7, 3.267) -- (7.1, 3.328) -- (7.2, 3.39) -- (7.3, 3.452) -- (7.4, 3.514) -- (7.5, 3.577) -- (7.6, 3.639) -- (7.7, 3.702) -- (7.8, 3.765) -- (7.9, 3.829) -- (8, 3.892) -- (8.1, 3.956) -- (8.2, 4.02) -- (8.3, 4.084) -- (8.4, 4.149) -- (8.5, 4.213) -- (8.6, 4.278) -- (8.7, 4.343) -- (8.8, 4.408) -- (8.9, 4.474) -- (9, 4.539) -- (9.1, 4.605) -- (9.2, 4.671) -- (9.3, 4.737) -- (9.4, 4.803) -- (9.5, 4.87) -- (9.6, 4.936) -- (9.7, 5.003) -- (9.8, 5.07) -- (9.9, 5.137) -- (10, 5.205);

% samples
\draw[black!50] (0, 0.007) -- (0.1, 0.019) -- (0.2, 0.031) -- (0.3, 0.044) -- (0.4, 0.058) -- (0.5, 0.073) -- (0.6, 0.088) -- (0.7, 0.103) -- (0.8, 0.12) -- (0.9, 0.137) -- (1, 0.155) -- (1.1, 0.173) -- (1.2, 0.192) -- (1.3, 0.211) -- (1.4, 0.231) -- (1.5, 0.251) -- (1.6, 0.272) -- (1.7, 0.294) -- (1.8, 0.316) -- (1.9, 0.338) -- (2, 0.361) -- (2.1, 0.385) -- (2.2, 0.408) -- (2.3, 0.433) -- (2.4, 0.457) -- (2.5, 0.482) -- (2.6, 0.508) -- (2.7, 0.534) -- (2.8, 0.56) -- (2.9, 0.587) -- (3, 0.613) -- (3.1, 0.641) -- (3.2, 0.668) -- (3.3, 0.696) -- (3.4, 0.724) -- (3.5, 0.753) -- (3.6, 0.782) -- (3.7, 0.811) -- (3.8, 0.84) -- (3.9, 0.87) -- (4, 0.899) -- (4.1, 0.93) -- (4.2, 0.96) -- (4.3, 0.99) -- (4.4, 1.021) -- (4.5, 1.052) -- (4.6, 1.083) -- (4.7, 1.115) -- (4.8, 1.146) -- (4.9, 1.178) -- (5, 1.21) -- (5.1, 1.242) -- (5.2, 1.275) -- (5.3, 1.307) -- (5.4, 1.34) -- (5.5, 1.372) -- (5.6, 1.405) -- (5.7, 1.438) -- (5.8, 1.472) -- (5.9, 1.505) -- (6, 1.538) -- (6.1, 1.572) -- (6.2, 1.606) -- (6.3, 1.639) -- (6.4, 1.673) -- (6.5, 1.707) -- (6.6, 1.741) -- (6.7, 1.776) -- (6.8, 1.81) -- (6.9, 1.844) -- (7, 1.879) -- (7.1, 1.913) -- (7.2, 1.948) -- (7.3, 1.982) -- (7.4, 2.017) -- (7.5, 2.052) -- (7.6, 2.087) -- (7.7, 2.122) -- (7.8, 2.157) -- (7.9, 2.192) -- (8, 2.227) -- (8.1, 2.262) -- (8.2, 2.297) -- (8.3, 2.332) -- (8.4, 2.367) -- (8.5, 2.403) -- (8.6, 2.438) -- (8.7, 2.473) -- (8.8, 2.509) -- (8.9, 2.544) -- (9, 2.58) -- (9.1, 2.615) -- (9.2, 2.651) -- (9.3, 2.686) -- (9.4, 2.722) -- (9.5, 2.757) -- (9.6, 2.793) -- (9.7, 2.828) -- (9.8, 2.864) -- (9.9, 2.899) -- (10, 2.935) node[right]{\footnotesize $28$};
\draw[black!50] (0, 0) -- (0.1, 0.016) -- (0.2, 0.032) -- (0.3, 0.05) -- (0.4, 0.069) -- (0.5, 0.089) -- (0.6, 0.11) -- (0.7, 0.132) -- (0.8, 0.155) -- (0.9, 0.179) -- (1, 0.205) -- (1.1, 0.231) -- (1.2, 0.259) -- (1.3, 0.287) -- (1.4, 0.317) -- (1.5, 0.347) -- (1.6, 0.378) -- (1.7, 0.411) -- (1.8, 0.444) -- (1.9, 0.478) -- (2, 0.513) -- (2.1, 0.549) -- (2.2, 0.585) -- (2.3, 0.623) -- (2.4, 0.661) -- (2.5, 0.7) -- (2.6, 0.739) -- (2.7, 0.78) -- (2.8, 0.821) -- (2.9, 0.862) -- (3, 0.905) -- (3.1, 0.948) -- (3.2, 0.991) -- (3.3, 1.035) -- (3.4, 1.08) -- (3.5, 1.125) -- (3.6, 1.17) -- (3.7, 1.216) -- (3.8, 1.263) -- (3.9, 1.31) -- (4, 1.357) -- (4.1, 1.405) -- (4.2, 1.453) -- (4.3, 1.502) -- (4.4, 1.55) -- (4.5, 1.6) -- (4.6, 1.649) -- (4.7, 1.699) -- (4.8, 1.749) -- (4.9, 1.8) -- (5, 1.85) -- (5.1, 1.901) -- (5.2, 1.952) -- (5.3, 2.004) -- (5.4, 2.056) -- (5.5, 2.107) -- (5.6, 2.159) -- (5.7, 2.212) -- (5.8, 2.264) -- (5.9, 2.317) -- (6, 2.37) -- (6.1, 2.423) -- (6.2, 2.476) -- (6.3, 2.529) -- (6.4, 2.582) -- (6.5, 2.636) -- (6.6, 2.689) -- (6.7, 2.743) -- (6.8, 2.797) -- (6.9, 2.851) -- (7, 2.905) -- (7.1, 2.959) -- (7.2, 3.013) -- (7.3, 3.068) -- (7.4, 3.122) -- (7.5, 3.177) -- (7.6, 3.231) -- (7.7, 3.286) -- (7.8, 3.34) -- (7.9, 3.395) -- (8, 3.45) -- (8.1, 3.505) -- (8.2, 3.56) -- (8.3, 3.615) -- (8.4, 3.67) -- (8.5, 3.725) -- (8.6, 3.78) -- (8.7, 3.835) -- (8.8, 3.89) -- (8.9, 3.946) -- (9, 4.001) -- (9.1, 4.056) -- (9.2, 4.112) -- (9.3, 4.167) -- (9.4, 4.223) -- (9.5, 4.278) -- (9.6, 4.334) -- (9.7, 4.389) -- (9.8, 4.445) -- (9.9, 4.5) -- (10, 4.556) node[right]{\footnotesize $21$};
\draw[black!50] (0, 0) -- (0.1, 0.019) -- (0.2, 0.039) -- (0.3, 0.06) -- (0.4, 0.083) -- (0.5, 0.107) -- (0.6, 0.132) -- (0.7, 0.158) -- (0.8, 0.186) -- (0.9, 0.214) -- (1, 0.244) -- (1.1, 0.275) -- (1.2, 0.307) -- (1.3, 0.34) -- (1.4, 0.374) -- (1.5, 0.409) -- (1.6, 0.444) -- (1.7, 0.481) -- (1.8, 0.519) -- (1.9, 0.557) -- (2, 0.597) -- (2.1, 0.637) -- (2.2, 0.678) -- (2.3, 0.719) -- (2.4, 0.762) -- (2.5, 0.805) -- (2.6, 0.849) -- (2.7, 0.893) -- (2.8, 0.939) -- (2.9, 0.984) -- (3, 1.031) -- (3.1, 1.078) -- (3.2, 1.126) -- (3.3, 1.174) -- (3.4, 1.222) -- (3.5, 1.272) -- (3.6, 1.321) -- (3.7, 1.372) -- (3.8, 1.422) -- (3.9, 1.473) -- (4, 1.525) -- (4.1, 1.577) -- (4.2, 1.63) -- (4.3, 1.682) -- (4.4, 1.736) -- (4.5, 1.789) -- (4.6, 1.843) -- (4.7, 1.897) -- (4.8, 1.952) -- (4.9, 2.007) -- (5, 2.062) -- (5.1, 2.118) -- (5.2, 2.174) -- (5.3, 2.23) -- (5.4, 2.286) -- (5.5, 2.343) -- (5.6, 2.4) -- (5.7, 2.457) -- (5.8, 2.514) -- (5.9, 2.572) -- (6, 2.63) -- (6.1, 2.688) -- (6.2, 2.746) -- (6.3, 2.804) -- (6.4, 2.863) -- (6.5, 2.921) -- (6.6, 2.98) -- (6.7, 3.039) -- (6.8, 3.099) -- (6.9, 3.158) -- (7, 3.217) -- (7.1, 3.277) -- (7.2, 3.337) -- (7.3, 3.397) -- (7.4, 3.456) -- (7.5, 3.517) -- (7.6, 3.577) -- (7.7, 3.637) -- (7.8, 3.697) -- (7.9, 3.758) -- (8, 3.818) -- (8.1, 3.879) -- (8.2, 3.94) -- (8.3, 4.001) -- (8.4, 4.061) -- (8.5, 4.122) -- (8.6, 4.183) -- (8.7, 4.244) -- (8.8, 4.306) -- (8.9, 4.367) -- (9, 4.428) -- (9.1, 4.489) -- (9.2, 4.551) -- (9.3, 4.612) -- (9.4, 4.673) -- (9.5, 4.735) -- (9.6, 4.796) -- (9.7, 4.858) -- (9.8, 4.92) -- (9.9, 4.981) -- (10, 5.043) node[right]{\footnotesize $34$};
\draw[black!50] (0, 0) -- (0.1, 0.012) -- (0.2, 0.025) -- (0.3, 0.039) -- (0.4, 0.053) -- (0.5, 0.069) -- (0.6, 0.086) -- (0.7, 0.103) -- (0.8, 0.121) -- (0.9, 0.141) -- (1, 0.161) -- (1.1, 0.182) -- (1.2, 0.203) -- (1.3, 0.226) -- (1.4, 0.249) -- (1.5, 0.273) -- (1.6, 0.298) -- (1.7, 0.323) -- (1.8, 0.349) -- (1.9, 0.376) -- (2, 0.403) -- (2.1, 0.431) -- (2.2, 0.46) -- (2.3, 0.489) -- (2.4, 0.519) -- (2.5, 0.549) -- (2.6, 0.58) -- (2.7, 0.612) -- (2.8, 0.643) -- (2.9, 0.676) -- (3, 0.709) -- (3.1, 0.742) -- (3.2, 0.775) -- (3.3, 0.809) -- (3.4, 0.844) -- (3.5, 0.878) -- (3.6, 0.914) -- (3.7, 0.949) -- (3.8, 0.985) -- (3.9, 1.021) -- (4, 1.057) -- (4.1, 1.094) -- (4.2, 1.131) -- (4.3, 1.168) -- (4.4, 1.206) -- (4.5, 1.243) -- (4.6, 1.281) -- (4.7, 1.32) -- (4.8, 1.358) -- (4.9, 1.397) -- (5, 1.436) -- (5.1, 1.475) -- (5.2, 1.514) -- (5.3, 1.554) -- (5.4, 1.593) -- (5.5, 1.633) -- (5.6, 1.673) -- (5.7, 1.713) -- (5.8, 1.754) -- (5.9, 1.794) -- (6, 1.835) -- (6.1, 1.875) -- (6.2, 1.916) -- (6.3, 1.957) -- (6.4, 1.998) -- (6.5, 2.04) -- (6.6, 2.081) -- (6.7, 2.122) -- (6.8, 2.164) -- (6.9, 2.206) -- (7, 2.247) -- (7.1, 2.289) -- (7.2, 2.331) -- (7.3, 2.373) -- (7.4, 2.415) -- (7.5, 2.457) -- (7.6, 2.499) -- (7.7, 2.541) -- (7.8, 2.584) -- (7.9, 2.626) -- (8, 2.668) -- (8.1, 2.711) -- (8.2, 2.753) -- (8.3, 2.796) -- (8.4, 2.838) -- (8.5, 2.881) -- (8.6, 2.924) -- (8.7, 2.966) -- (8.8, 3.009) -- (8.9, 3.052) -- (9, 3.095) -- (9.1, 3.137) -- (9.2, 3.18) -- (9.3, 3.223) -- (9.4, 3.266) -- (9.5, 3.309) -- (9.6, 3.352) -- (9.7, 3.395) -- (9.8, 3.437) -- (9.9, 3.48) -- (10, 3.523) node[right]{\footnotesize $27$};
\draw[black!50] (0, 0) -- (0.1, 0.026) -- (0.2, 0.053) -- (0.3, 0.083) -- (0.4, 0.113) -- (0.5, 0.146) -- (0.6, 0.18) -- (0.7, 0.215) -- (0.8, 0.252) -- (0.9, 0.29) -- (1, 0.329) -- (1.1, 0.37) -- (1.2, 0.411) -- (1.3, 0.454) -- (1.4, 0.498) -- (1.5, 0.542) -- (1.6, 0.588) -- (1.7, 0.635) -- (1.8, 0.682) -- (1.9, 0.73) -- (2, 0.779) -- (2.1, 0.828) -- (2.2, 0.879) -- (2.3, 0.929) -- (2.4, 0.981) -- (2.5, 1.033) -- (2.6, 1.085) -- (2.7, 1.138) -- (2.8, 1.192) -- (2.9, 1.246) -- (3, 1.3) -- (3.1, 1.355) -- (3.2, 1.41) -- (3.3, 1.465) -- (3.4, 1.521) -- (3.5, 1.577) -- (3.6, 1.633) -- (3.7, 1.69) -- (3.8, 1.747) -- (3.9, 1.804) -- (4, 1.861) -- (4.1, 1.919) -- (4.2, 1.977) -- (4.3, 2.035) -- (4.4, 2.093) -- (4.5, 2.152) -- (4.6, 2.21) -- (4.7, 2.269) -- (4.8, 2.328) -- (4.9, 2.387) -- (5, 2.446) -- (5.1, 2.506) -- (5.2, 2.565) -- (5.3, 2.625) -- (5.4, 2.684) -- (5.5, 2.744) -- (5.6, 2.804) -- (5.7, 2.864) -- (5.8, 2.924) -- (5.9, 2.984) -- (6, 3.045) -- (6.1, 3.105) -- (6.2, 3.165) -- (6.3, 3.226) -- (6.4, 3.286) -- (6.5, 3.347) -- (6.6, 3.408) -- (6.7, 3.468) -- (6.8, 3.529) -- (6.9, 3.59) -- (7, 3.651) -- (7.1, 3.712) -- (7.2, 3.773) -- (7.3, 3.834) -- (7.4, 3.895) -- (7.5, 3.956) -- (7.6, 4.017) -- (7.7, 4.078) -- (7.8, 4.139) -- (7.9, 4.2) -- (8, 4.261) -- (8.1, 4.322) -- (8.2, 4.384) -- (8.3, 4.445) -- (8.4, 4.506) -- (8.5, 4.567) -- (8.6, 4.629) -- (8.7, 4.69) -- (8.8, 4.751) -- (8.9, 4.813) -- (9, 4.874) -- (9.1, 4.935) -- (9.2, 4.997) -- (9.3, 5.058) -- (9.4, 5.119) -- (9.5, 5.181) -- (9.6, 5.242) -- (9.7, 5.304) -- (9.8, 5.365) -- (9.9, 5.426) -- (10, 5.488) node[right]{\footnotesize $29$};
\end{tikzpicture}
\caption{5 realizations of $\hat\varphi_n^\Delta$ using $\xi = \Delta = 1$ and $N\Delta=25$. Number of probes $n$ are given on the right.}
\label{f:nonBootstrap}
\end{minipage}\hfill
\begin{minipage}{0.47\textwidth}
\centering
\begin{tikzpicture}[scale=.66]
\draw[<->] (0, 6) -- (0, 0) -- (10, 0) node[right]{$\alpha$};

\foreach \x in {2, 4, 6, 8}
{
\draw (\x, .05) -- (\x, -.05) node[below]{$\x$};
}
\foreach \y in {2, 4}
{
\draw (.05, \y) -- (-.05, \y) node[left]{$\y$};
}

\draw (1, 4.25) -- (4, 4.25) -- (4, 6) -- (1, 6) -- (1, 4.25);

\draw[black!50] (1.5, 5.5) -- (2.5, 5.5) node[black, right]{$\hat \varphi_{\text r, K}^\Delta$};
\draw           (1.5, 4.75) -- (2.5, 4.75) node[black, right]{$\varphi_\epsilon$};

% actual
\draw (0, 0) -- (0.1, 0.021) -- (0.2, 0.043) -- (0.3, 0.067) -- (0.4, 0.092) -- (0.5, 0.118) -- (0.6, 0.145) -- (0.7, 0.174) -- (0.8, 0.203) -- (0.9, 0.234) -- (1, 0.266) -- (1.1, 0.298) -- (1.2, 0.332) -- (1.3, 0.366) -- (1.4, 0.402) -- (1.5, 0.438) -- (1.6, 0.475) -- (1.7, 0.513) -- (1.8, 0.551) -- (1.9, 0.59) -- (2, 0.631) -- (2.1, 0.671) -- (2.2, 0.713) -- (2.3, 0.755) -- (2.4, 0.798) -- (2.5, 0.841) -- (2.6, 0.885) -- (2.7, 0.93) -- (2.8, 0.975) -- (2.9, 1.021) -- (3, 1.068) -- (3.1, 1.115) -- (3.2, 1.162) -- (3.3, 1.21) -- (3.4, 1.259) -- (3.5, 1.308) -- (3.6, 1.357) -- (3.7, 1.407) -- (3.8, 1.457) -- (3.9, 1.508) -- (4, 1.56) -- (4.1, 1.611) -- (4.2, 1.664) -- (4.3, 1.716) -- (4.4, 1.769) -- (4.5, 1.823) -- (4.6, 1.876) -- (4.7, 1.931) -- (4.8, 1.985) -- (4.9, 2.04) -- (5, 2.095) -- (5.1, 2.151) -- (5.2, 2.207) -- (5.3, 2.263) -- (5.4, 2.32) -- (5.5, 2.377) -- (5.6, 2.434) -- (5.7, 2.492) -- (5.8, 2.55) -- (5.9, 2.608) -- (6, 2.666) -- (6.1, 2.725) -- (6.2, 2.784) -- (6.3, 2.844) -- (6.4, 2.903) -- (6.5, 2.963) -- (6.6, 3.023) -- (6.7, 3.084) -- (6.8, 3.145) -- (6.9, 3.206) -- (7, 3.267) -- (7.1, 3.328) -- (7.2, 3.39) -- (7.3, 3.452) -- (7.4, 3.514) -- (7.5, 3.577) -- (7.6, 3.639) -- (7.7, 3.702) -- (7.8, 3.765) -- (7.9, 3.829) -- (8, 3.892) -- (8.1, 3.956) -- (8.2, 4.02) -- (8.3, 4.084) -- (8.4, 4.149) -- (8.5, 4.213) -- (8.6, 4.278) -- (8.7, 4.343) -- (8.8, 4.408) -- (8.9, 4.474) -- (9, 4.539) -- (9.1, 4.605) -- (9.2, 4.671) -- (9.3, 4.737) -- (9.4, 4.803) -- (9.5, 4.87) -- (9.6, 4.936) -- (9.7, 5.003) -- (9.8, 5.07) -- (9.9, 5.137) -- (10, 5.205);

% delta = 1
\draw[black!50] (0, 0.005) -- (0.1, 0.033) -- (0.2, 0.062) -- (0.3, 0.091) -- (0.4, 0.122) -- (0.5, 0.154) -- (0.6, 0.186) -- (0.7, 0.219) -- (0.8, 0.254) -- (0.9, 0.289) -- (1, 0.325) -- (1.1, 0.362) -- (1.2, 0.399) -- (1.3, 0.438) -- (1.4, 0.477) -- (1.5, 0.517) -- (1.6, 0.557) -- (1.7, 0.599) -- (1.8, 0.641) -- (1.9, 0.684) -- (2, 0.727) -- (2.1, 0.771) -- (2.2, 0.816) -- (2.3, 0.861) -- (2.4, 0.907) -- (2.5, 0.954) -- (2.6, 1.001) -- (2.7, 1.049) -- (2.8, 1.097) -- (2.9, 1.146) -- (3, 1.195) -- (3.1, 1.245) -- (3.2, 1.296) -- (3.3, 1.346) -- (3.4, 1.398) -- (3.5, 1.45) -- (3.6, 1.502) -- (3.7, 1.554) -- (3.8, 1.607) -- (3.9, 1.661) -- (4, 1.715) -- (4.1, 1.769) -- (4.2, 1.824) -- (4.3, 1.879) -- (4.4, 1.934) -- (4.5, 1.99) -- (4.6, 2.046) -- (4.7, 2.103) -- (4.8, 2.159) -- (4.9, 2.217) -- (5, 2.274) -- (5.1, 2.332) -- (5.2, 2.39) -- (5.3, 2.448) -- (5.4, 2.507) -- (5.5, 2.566) -- (5.6, 2.625) -- (5.7, 2.685) -- (5.8, 2.744) -- (5.9, 2.805) -- (6, 2.865) -- (6.1, 2.925) -- (6.2, 2.986) -- (6.3, 3.047) -- (6.4, 3.109) -- (6.5, 3.17) -- (6.6, 3.232) -- (6.7, 3.294) -- (6.8, 3.356) -- (6.9, 3.419) -- (7, 3.481) -- (7.1, 3.544) -- (7.2, 3.607) -- (7.3, 3.671) -- (7.4, 3.734) -- (7.5, 3.798) -- (7.6, 3.862) -- (7.7, 3.926) -- (7.8, 3.99) -- (7.9, 4.055) -- (8, 4.119) -- (8.1, 4.184) -- (8.2, 4.249) -- (8.3, 4.314) -- (8.4, 4.38) -- (8.5, 4.445) -- (8.6, 4.511) -- (8.7, 4.577) -- (8.8, 4.643) -- (8.9, 4.709) -- (9, 4.775) -- (9.1, 4.842) -- (9.2, 4.908) -- (9.3, 4.975) -- (9.4, 5.042) -- (9.5, 5.109) -- (9.6, 5.177) -- (9.7, 5.244) -- (9.8, 5.311) -- (9.9, 5.379) -- (10, 5.447) node[right]{\footnotesize $1$};
\draw[black!50] (0, 0.005) -- (0.1, 0.032) -- (0.2, 0.061) -- (0.3, 0.09) -- (0.4, 0.121) -- (0.5, 0.152) -- (0.6, 0.184) -- (0.7, 0.217) -- (0.8, 0.251) -- (0.9, 0.286) -- (1, 0.321) -- (1.1, 0.357) -- (1.2, 0.395) -- (1.3, 0.433) -- (1.4, 0.471) -- (1.5, 0.511) -- (1.6, 0.551) -- (1.7, 0.592) -- (1.8, 0.634) -- (1.9, 0.676) -- (2, 0.719) -- (2.1, 0.763) -- (2.2, 0.807) -- (2.3, 0.852) -- (2.4, 0.897) -- (2.5, 0.943) -- (2.6, 0.99) -- (2.7, 1.037) -- (2.8, 1.085) -- (2.9, 1.133) -- (3, 1.182) -- (3.1, 1.231) -- (3.2, 1.281) -- (3.3, 1.331) -- (3.4, 1.382) -- (3.5, 1.433) -- (3.6, 1.485) -- (3.7, 1.537) -- (3.8, 1.589) -- (3.9, 1.642) -- (4, 1.696) -- (4.1, 1.749) -- (4.2, 1.803) -- (4.3, 1.858) -- (4.4, 1.913) -- (4.5, 1.968) -- (4.6, 2.023) -- (4.7, 2.079) -- (4.8, 2.136) -- (4.9, 2.192) -- (5, 2.249) -- (5.1, 2.306) -- (5.2, 2.364) -- (5.3, 2.421) -- (5.4, 2.479) -- (5.5, 2.538) -- (5.6, 2.596) -- (5.7, 2.655) -- (5.8, 2.714) -- (5.9, 2.774) -- (6, 2.834) -- (6.1, 2.894) -- (6.2, 2.954) -- (6.3, 3.014) -- (6.4, 3.075) -- (6.5, 3.136) -- (6.6, 3.197) -- (6.7, 3.259) -- (6.8, 3.32) -- (6.9, 3.382) -- (7, 3.444) -- (7.1, 3.506) -- (7.2, 3.569) -- (7.3, 3.632) -- (7.4, 3.695) -- (7.5, 3.758) -- (7.6, 3.821) -- (7.7, 3.885) -- (7.8, 3.948) -- (7.9, 4.012) -- (8, 4.076) -- (8.1, 4.141) -- (8.2, 4.205) -- (8.3, 4.27) -- (8.4, 4.334) -- (8.5, 4.399) -- (8.6, 4.465) -- (8.7, 4.53) -- (8.8, 4.595) -- (8.9, 4.661) -- (9, 4.727) -- (9.1, 4.793) -- (9.2, 4.859) -- (9.3, 4.925) -- (9.4, 4.992) -- (9.5, 5.058) -- (9.6, 5.125) -- (9.7, 5.192) -- (9.8, 5.259) -- (9.9, 5.326) -- (10, 5.393);

% delta = .1
\draw[black!50] (0, 0.007) -- (0.1, 0.028) -- (0.2, 0.05) -- (0.3, 0.073) -- (0.4, 0.096) -- (0.5, 0.12) -- (0.6, 0.144) -- (0.7, 0.17) -- (0.8, 0.195) -- (0.9, 0.222) -- (1, 0.248) -- (1.1, 0.276) -- (1.2, 0.304) -- (1.3, 0.332) -- (1.4, 0.361) -- (1.5, 0.391) -- (1.6, 0.421) -- (1.7, 0.451) -- (1.8, 0.482) -- (1.9, 0.514) -- (2, 0.546) -- (2.1, 0.578) -- (2.2, 0.611) -- (2.3, 0.645) -- (2.4, 0.678) -- (2.5, 0.712) -- (2.6, 0.747) -- (2.7, 0.782) -- (2.8, 0.817) -- (2.9, 0.853) -- (3, 0.889) -- (3.1, 0.926) -- (3.2, 0.963) -- (3.3, 1) -- (3.4, 1.038) -- (3.5, 1.076) -- (3.6, 1.114) -- (3.7, 1.153) -- (3.8, 1.192) -- (3.9, 1.231) -- (4, 1.271) -- (4.1, 1.311) -- (4.2, 1.351) -- (4.3, 1.392) -- (4.4, 1.433) -- (4.5, 1.474) -- (4.6, 1.516) -- (4.7, 1.557) -- (4.8, 1.599) -- (4.9, 1.642) -- (5, 1.684) -- (5.1, 1.727) -- (5.2, 1.771) -- (5.3, 1.814) -- (5.4, 1.858) -- (5.5, 1.902) -- (5.6, 1.946) -- (5.7, 1.99) -- (5.8, 2.035) -- (5.9, 2.08) -- (6, 2.125) -- (6.1, 2.17) -- (6.2, 2.216) -- (6.3, 2.262) -- (6.4, 2.308) -- (6.5, 2.354) -- (6.6, 2.401) -- (6.7, 2.447) -- (6.8, 2.494) -- (6.9, 2.541) -- (7, 2.589) -- (7.1, 2.636) -- (7.2, 2.684) -- (7.3, 2.732) -- (7.4, 2.78) -- (7.5, 2.828) -- (7.6, 2.877) -- (7.7, 2.926) -- (7.8, 2.974) -- (7.9, 3.023) -- (8, 3.073) -- (8.1, 3.122) -- (8.2, 3.172) -- (8.3, 3.221) -- (8.4, 3.271) -- (8.5, 3.321) -- (8.6, 3.371) -- (8.7, 3.422) -- (8.8, 3.472) -- (8.9, 3.523) -- (9, 3.574) -- (9.1, 3.625) -- (9.2, 3.676) -- (9.3, 3.727) -- (9.4, 3.779) -- (9.5, 3.83) -- (9.6, 3.882) -- (9.7, 3.934) -- (9.8, 3.986) -- (9.9, 4.038) -- (10, 4.091) node[right]{\footnotesize $0.1$};
\draw[black!50] (0, 0.007) -- (0.1, 0.029) -- (0.2, 0.052) -- (0.3, 0.075) -- (0.4, 0.1) -- (0.5, 0.124) -- (0.6, 0.15) -- (0.7, 0.176) -- (0.8, 0.203) -- (0.9, 0.23) -- (1, 0.258) -- (1.1, 0.286) -- (1.2, 0.315) -- (1.3, 0.345) -- (1.4, 0.375) -- (1.5, 0.405) -- (1.6, 0.437) -- (1.7, 0.468) -- (1.8, 0.5) -- (1.9, 0.533) -- (2, 0.566) -- (2.1, 0.6) -- (2.2, 0.634) -- (2.3, 0.668) -- (2.4, 0.703) -- (2.5, 0.739) -- (2.6, 0.775) -- (2.7, 0.811) -- (2.8, 0.847) -- (2.9, 0.885) -- (3, 0.922) -- (3.1, 0.96) -- (3.2, 0.998) -- (3.3, 1.037) -- (3.4, 1.076) -- (3.5, 1.115) -- (3.6, 1.155) -- (3.7, 1.195) -- (3.8, 1.235) -- (3.9, 1.276) -- (4, 1.317) -- (4.1, 1.359) -- (4.2, 1.4) -- (4.3, 1.442) -- (4.4, 1.485) -- (4.5, 1.528) -- (4.6, 1.571) -- (4.7, 1.614) -- (4.8, 1.657) -- (4.9, 1.701) -- (5, 1.745) -- (5.1, 1.79) -- (5.2, 1.834) -- (5.3, 1.879) -- (5.4, 1.925) -- (5.5, 1.97) -- (5.6, 2.016) -- (5.7, 2.062) -- (5.8, 2.108) -- (5.9, 2.154) -- (6, 2.201) -- (6.1, 2.248) -- (6.2, 2.295) -- (6.3, 2.343) -- (6.4, 2.39) -- (6.5, 2.438) -- (6.6, 2.486) -- (6.7, 2.534) -- (6.8, 2.583) -- (6.9, 2.631) -- (7, 2.68) -- (7.1, 2.729) -- (7.2, 2.779) -- (7.3, 2.828) -- (7.4, 2.878) -- (7.5, 2.928) -- (7.6, 2.978) -- (7.7, 3.028) -- (7.8, 3.078) -- (7.9, 3.129) -- (8, 3.18) -- (8.1, 3.231) -- (8.2, 3.282) -- (8.3, 3.333) -- (8.4, 3.385) -- (8.5, 3.436) -- (8.6, 3.488) -- (8.7, 3.54) -- (8.8, 3.592) -- (8.9, 3.644) -- (9, 3.697) -- (9.1, 3.749) -- (9.2, 3.802) -- (9.3, 3.855) -- (9.4, 3.908) -- (9.5, 3.961) -- (9.6, 4.014) -- (9.7, 4.068) -- (9.8, 4.122) -- (9.9, 4.175) -- (10, 4.229);

% delta = .01
\draw[black!50] (0, 0.007) -- (0.1, 0.031) -- (0.2, 0.056) -- (0.3, 0.082) -- (0.4, 0.109) -- (0.5, 0.136) -- (0.6, 0.164) -- (0.7, 0.193) -- (0.8, 0.222) -- (0.9, 0.252) -- (1, 0.283) -- (1.1, 0.314) -- (1.2, 0.346) -- (1.3, 0.378) -- (1.4, 0.411) -- (1.5, 0.445) -- (1.6, 0.479) -- (1.7, 0.514) -- (1.8, 0.55) -- (1.9, 0.585) -- (2, 0.622) -- (2.1, 0.659) -- (2.2, 0.696) -- (2.3, 0.734) -- (2.4, 0.773) -- (2.5, 0.811) -- (2.6, 0.851) -- (2.7, 0.891) -- (2.8, 0.931) -- (2.9, 0.972) -- (3, 1.013) -- (3.1, 1.054) -- (3.2, 1.096) -- (3.3, 1.139) -- (3.4, 1.181) -- (3.5, 1.224) -- (3.6, 1.268) -- (3.7, 1.312) -- (3.8, 1.356) -- (3.9, 1.401) -- (4, 1.446) -- (4.1, 1.491) -- (4.2, 1.537) -- (4.3, 1.583) -- (4.4, 1.629) -- (4.5, 1.676) -- (4.6, 1.723) -- (4.7, 1.77) -- (4.8, 1.818) -- (4.9, 1.866) -- (5, 1.914) -- (5.1, 1.963) -- (5.2, 2.011) -- (5.3, 2.061) -- (5.4, 2.11) -- (5.5, 2.16) -- (5.6, 2.209) -- (5.7, 2.26) -- (5.8, 2.31) -- (5.9, 2.361) -- (6, 2.412) -- (6.1, 2.463) -- (6.2, 2.514) -- (6.3, 2.566) -- (6.4, 2.618) -- (6.5, 2.67) -- (6.6, 2.723) -- (6.7, 2.775) -- (6.8, 2.828) -- (6.9, 2.881) -- (7, 2.934) -- (7.1, 2.988) -- (7.2, 3.042) -- (7.3, 3.096) -- (7.4, 3.15) -- (7.5, 3.204) -- (7.6, 3.258) -- (7.7, 3.313) -- (7.8, 3.368) -- (7.9, 3.423) -- (8, 3.479) -- (8.1, 3.534) -- (8.2, 3.59) -- (8.3, 3.645) -- (8.4, 3.702) -- (8.5, 3.758) -- (8.6, 3.814) -- (8.7, 3.871) -- (8.8, 3.927) -- (8.9, 3.984) -- (9, 4.041) -- (9.1, 4.098) -- (9.2, 4.156) -- (9.3, 4.213) -- (9.4, 4.271) -- (9.5, 4.329) -- (9.6, 4.387) -- (9.7, 4.445) -- (9.8, 4.503) -- (9.9, 4.562) -- (10, 4.62);
\draw[black!50] (0, 0.007) -- (0.1, 0.031) -- (0.2, 0.057) -- (0.3, 0.083) -- (0.4, 0.11) -- (0.5, 0.137) -- (0.6, 0.165) -- (0.7, 0.194) -- (0.8, 0.224) -- (0.9, 0.254) -- (1, 0.285) -- (1.1, 0.317) -- (1.2, 0.349) -- (1.3, 0.382) -- (1.4, 0.415) -- (1.5, 0.449) -- (1.6, 0.484) -- (1.7, 0.519) -- (1.8, 0.555) -- (1.9, 0.591) -- (2, 0.628) -- (2.1, 0.665) -- (2.2, 0.703) -- (2.3, 0.741) -- (2.4, 0.78) -- (2.5, 0.819) -- (2.6, 0.859) -- (2.7, 0.899) -- (2.8, 0.94) -- (2.9, 0.981) -- (3, 1.022) -- (3.1, 1.064) -- (3.2, 1.107) -- (3.3, 1.149) -- (3.4, 1.193) -- (3.5, 1.236) -- (3.6, 1.28) -- (3.7, 1.324) -- (3.8, 1.369) -- (3.9, 1.414) -- (4, 1.46) -- (4.1, 1.505) -- (4.2, 1.552) -- (4.3, 1.598) -- (4.4, 1.645) -- (4.5, 1.692) -- (4.6, 1.739) -- (4.7, 1.787) -- (4.8, 1.835) -- (4.9, 1.884) -- (5, 1.932) -- (5.1, 1.981) -- (5.2, 2.031) -- (5.3, 2.08) -- (5.4, 2.13) -- (5.5, 2.18) -- (5.6, 2.231) -- (5.7, 2.281) -- (5.8, 2.332) -- (5.9, 2.383) -- (6, 2.435) -- (6.1, 2.486) -- (6.2, 2.538) -- (6.3, 2.591) -- (6.4, 2.643) -- (6.5, 2.696) -- (6.6, 2.749) -- (6.7, 2.802) -- (6.8, 2.855) -- (6.9, 2.909) -- (7, 2.962) -- (7.1, 3.016) -- (7.2, 3.071) -- (7.3, 3.125) -- (7.4, 3.18) -- (7.5, 3.234) -- (7.6, 3.289) -- (7.7, 3.345) -- (7.8, 3.4) -- (7.9, 3.456) -- (8, 3.511) -- (8.1, 3.567) -- (8.2, 3.624) -- (8.3, 3.68) -- (8.4, 3.737) -- (8.5, 3.793) -- (8.6, 3.85) -- (8.7, 3.907) -- (8.8, 3.964) -- (8.9, 4.022) -- (9, 4.079) -- (9.1, 4.137) -- (9.2, 4.195) -- (9.3, 4.253) -- (9.4, 4.311) -- (9.5, 4.37) -- (9.6, 4.428) -- (9.7, 4.487) -- (9.8, 4.545) -- (9.9, 4.604) -- (10, 4.664) node[right]{\footnotesize $0.01$};
\end{tikzpicture}
\caption{realizations of $\hat \varphi_{\text r, K}^\Delta$ using $\xi = 1$ and $K=1000$, for various $\Delta$ (on the right) and $N\Delta=25$. }
\label{f:bootstrapDelta}
\end{minipage}
\end{center}\end{figure}
We proceed by testing the performance of the resampling estimator.
Figure \ref{f:bootstrapDelta} shows two realizations of $\hat \varphi_{\text r, K}^\Delta$ each for three values of $\Delta$,  all based on the same simulated workload process. Per value of $\Delta$, the pairs of realizations are close, illustrating that the variance due to the random sampling is reduced significantly by applying the resampling procedure. It varies across simulations of the workload process for which value of $\Delta$ the estimator is closest to the true $\varphi$, but variations are mostly due to the stochastic nature of the workload process. In general taking smaller $\Delta$ is best, for it reduces bias and variance.

\bibliographystyle{apalike}
\bibliography{bib.bib}

\end{document}